\documentclass[10pt,a4paper,appendixprefix=true]{scrartcl}

\usepackage{amssymb}
\usepackage{amsmath}
\usepackage{amsfonts}
\usepackage{bbm}
\usepackage{amsthm}
\usepackage{mathrsfs}	
\usepackage[hidelinks]{hyperref}\usepackage{color}
\usepackage[margin=2.5cm]{geometry}
\usepackage[all,cmtip]{xy}
\usepackage[utf8]{inputenc}
\usepackage{graphicx}

\usepackage{varwidth}

\let\emph\undefined\newcommand{\emph}[1]{\textsl{#1}}
\usepackage{upgreek}
\usepackage{rotating}
\usepackage{tikz}
\usepackage{accents}
\usetikzlibrary{matrix,arrows,decorations.pathmorphing,shapes.geometric}
\usepackage{tikz-cd}
\usepackage{needspace}
\newcommand{\spaceplease}{\needspace{5\baselineskip}}

\usetikzlibrary{decorations.markings}
\usetikzlibrary{backgrounds}

\usepackage{etex}
\usetikzlibrary{svg.path}

\tikzstyle{tikzfig}=[baseline=-0.25em,scale=0.5]

\pgfkeys{/tikz/tikzit fill/.initial=0}
\pgfkeys{/tikz/tikzit draw/.initial=0}
\pgfkeys{/tikz/tikzit shape/.initial=0}
\pgfkeys{/tikz/tikzit category/.initial=0}

\pgfdeclarelayer{edgelayer}
\pgfdeclarelayer{nodelayer}
\pgfsetlayers{background,edgelayer,nodelayer,main}

\tikzstyle{none}=[inner sep=0mm]

\newcommand{\tikzfig}[1]{%
	{\tikzstyle{every picture}=[tikzfig]
		\IfFileExists{#1.tikz}
		{\input{#1.tikz}}
		{%
			\IfFileExists{./figures/#1.tikz}
			{\input{./figures/#1.tikz}}
			{\tikz[baseline=-0.5em]{\node[draw=red,font=\color{red},fill=red!10!white] {\textit{#1}};}}%
	}}%
}

\tikzstyle{every loop}=[]

\usepackage{tikzit}
\tikzstyle{black dot}=[fill=black, draw=black, shape=circle, minimum size=3pt, inner sep=0pt]
\tikzstyle{black dot small}=[fill=black, draw=black, shape=circle, minimum size=3pt, inner sep=0pt]
\tikzstyle{wbox}=[fill=white, draw=black, shape=rectangle, minimum height=0.5cm, minimum width=0.01cm]
\tikzstyle{bbox}=[fill=white, draw=blue, shape=rectangle, minimum height=0.5cm, minimum width=0.01cm]
\tikzstyle{rbox}=[fill=white, draw=red, shape=rectangle, minimum height=0.5cm, minimum width=0.01cm]
\tikzstyle{bwbox}=[draw=blue, shape=rectangle, minimum width=2cm, minimum height=0.5cm]
\tikzstyle{bbwbox}=[draw=blue, shape=rectangle, minimum width=1cm, minimum height=1cm]
\tikzstyle{big white circle}=[fill=white, draw=black, shape=circle, minimum width=0.75cm]
\tikzstyle{white dot big}=[fill=white, draw=black, shape=circle, inner sep=1pt]
\tikzstyle{white dot}=[fill=white, draw=black, shape=circle, minimum size=3pt, inner sep=0pt]
\tikzstyle{flat box}=[fill=white, draw=black, shape=rectangle, minimum width=1.3cm, minimum height=0.5cm,fill=morphismcolor]
\tikzstyle{square}=[fill=white, draw=black, shape=rectangle]
\tikzstyle{flat box 2}=[fill=white, draw=black, shape=rectangle, minimum height=0.5cm, minimum width=0.01cm,fill=morphismcolor]
\tikzstyle{bigbox}=[fill=white, draw=black, shape=rectangle, minimum height=0.5cm, minimum width=0.8cm,fill=morphismcolor]
\tikzstyle{over }=[front]
\tikzstyle{theta}=[fill=blue, draw=blue, shape=ellipse, minimum height=6pt, minimum width=6pt, inner sep=0pt]
\tikzstyle{thetabig}=[fill=blue, draw=blue, shape=ellipse, minimum width=1cm, minimum height=0.01cm]
\tikzstyle{thetainv}=[fill=blue, draw=red, shape=ellipse, minimum height=6pt, minimum width=6pt, inner sep=0pt]
\tikzstyle{thetabinv}=[fill=blue, draw=red, shape=ellipse, minimum width=1cm, minimum height=0.01cm]
\tikzstyle{bigdisk}=[draw=black, shape=circle, minimum width=3cm]
\tikzstyle{wdisk}=[shape=circle, minimum width=0.48cm,fill=white]
\tikzstyle{bigdisk2}=[draw=black, fill=lightgray, shape=circle, minimum width=3cm]
\tikzstyle{little disk}=[fill=white, draw=black, shape=circle, minimum width=0.5cm]

\tikzstyle{mid arrow}=[-, postaction={on each segment={mid arrow}}]
\tikzstyle{end arrow}=[->]
\tikzstyle{mover}=[-, link]
\tikzstyle{mydots}=[-,dotted]
\tikzstyle{open}=[-, line width=1pt,draw=blue,postaction={on each segment={mid arrow}}]
\tikzstyle{thick}=[-,line width=1pt]
\tikzstyle{dotarrow}=[->,dotted,draw=red,line width=1pt]
\tikzstyle{red mid arrow}=[-, draw={rgb,255: red,214; green,42; blue,51}, postaction={on each segment={mid arrow}}, line width=1pt]
\tikzstyle{RED}=[-, draw={rgb,255: red,214; green,42; blue,51}]
\tikzstyle{REDdotted}=[-,dotted, draw={rgb,255: red,214; green,42; blue,51}]
\tikzstyle{blue}=[-, draw=blue]
\tikzstyle{blue mid arrow}=[-, draw={rgb,255: red,23; green,37; blue,167}, postaction={on each segment={mid arrow}}, line width=1pt]
\tikzstyle{over}=[-, link]
\tikzstyle{mover}=[-, link]
\tikzstyle{mapsto}=[{|->}]

\usetikzlibrary{decorations.pathreplacing,decorations.markings}
\tikzset{
	on each segment/.style={
		decorate,
		decoration={
			show path construction,
			moveto code={},
			lineto code={
				\path [#1]
				(\tikzinputsegmentfirst) -- (\tikzinputsegmentlast);
			},
			curveto code={
				\path [#1] (\tikzinputsegmentfirst)
				.. controls
				(\tikzinputsegmentsupporta) and (\tikzinputsegmentsupportb)
				..
				(\tikzinputsegmentlast);
			},
			closepath code={
				\path [#1]
				(\tikzinputsegmentfirst) -- (\tikzinputsegmentlast);
			},
		},
	},
	mid arrow/.style={postaction={decorate,decoration={
				markings,
				mark=at position .7 with {\arrow[#1]{stealth}}
	}}},
}
\tikzset{%
	link/.style    = { white, double = black, line width = 1.8pt,
		double distance = 0.4pt },
	channel/.style = { white, double = black, line width = 0.8pt,
		double distance = 0.8pt },
}

\usepackage{mathtools}
\mathtoolsset{showonlyrefs}
\usepackage{epstopdf}

\newtheoremstyle{mytheorem}% name of the style to be used
  {\topsep}% measure of space to leave above the theorem. E.g.: 3pt
  {\topsep}% measure of space to leave below the theorem. E.g.: 3pt
  {\slshape}% name of font to use in the body of the theorem
  {0pt}% measure of space to indent
  {\scshape}% name of head font
  {.}% punctuation between head and body
  { }% space after theorem head; " " = normal interword space
  {\thmname{#1}\thmnumber{ #2}\thmnote{ {\normalfont\slshape(#3)}}}
  
  \newtheoremstyle{mydefinition}% name of the style to be used
    {\topsep}% measure of space to leave above the theorem. E.g.: 3pt
    {\topsep}% measure of space to leave below the theorem. E.g.: 3pt
    {\normalfont}% name of font to use in the body of the theorem
    {0pt}% measure of space to indent
    {\scshape}% name of head font
    {.}% punctuation between head and body
    { }% space after theorem head; " " = normal interword space
    {\thmname{#1}\thmnumber{ #2}\thmnote{ {\normalfont\slshape(#3)}}}

\theoremstyle{mytheorem}
\newtheorem{theorem}{Theorem}[section]

\makeatletter
\newtheorem*{rep@theorem}{\rep@title}
\newcommand{\newreptheorem}[2]{%
	\newenvironment{rep#1}[1]{%
		\def\rep@title{#2 \ref{##1}}%
		\begin{rep@theorem}}%
		{\end{rep@theorem}}}
\makeatother

\newreptheorem{theorem}{Theorem}
\newreptheorem{corollary}{Corollary}
\newtheorem{lemma}[theorem]{Lemma}
\newtheorem{proposition}[theorem]{Proposition}
\newtheorem{corollary}[theorem]{Corollary}
\theoremstyle{mydefinition}
\newtheorem{definition}[theorem]{Definition}

\newtheorem{remm}[theorem]{Remark}
\newenvironment{remark}
{\pushQED{\qed}\remm}
{\popQED\endremm}
\numberwithin{equation}{section}
\usepackage{enumitem}

\newenvironment{pnum}{\begin{enumerate}[topsep=2pt,parsep=2pt,partopsep=2pt,itemsep=0pt,label={(\roman{*})}]}{\end{enumerate}}

\usepackage{etoolbox}
\makeatletter
\g@addto@macro{\appendix}{%
	\patchcmd{\@@makechapterhead}% <cmd>
	{\endgraf\nobreak\vskip.5\baselineskip}% <search>
	{\hspace*{-.5em}:\space}% <replace>
	{}{}% <success><failure>
	\patchcmd{\@chapter}% <cmd>
	{\addchaptertocentry{\thechapter}}% <search>
	{\addchaptertocentry{Appendix~\thechapter:}}% <replace>
	{}{}% <success><failure>
	\addtocontents{toc}{%
		\protect\patchcmd{\protect\l@chapter}% <cmd>
		{1.5em}% <search>
		{6.5em}% <replace>
		{}{}}% <success><failure>
}
\makeatother

\DeclareMathSymbol{\Phiit}{\mathalpha}{letters}{"08}\let\Phi\undefined\newcommand{\Phi}{\Phiit}
\DeclareMathSymbol{\Psiit}{\mathalpha}{letters}{"09}\let\Psi\undefined\newcommand{\Psi}{\Psiit}
\DeclareMathSymbol{\Sigmait}{\mathalpha}{letters}{"06}\let\Sigma\undefined\newcommand{\Sigma}{\Sigmait}
\DeclareMathSymbol{\Xiit}{\mathalpha}{letters}{"04}
\DeclareMathSymbol{\Lambdait}{\mathalpha}{letters}{"03}\let\Lambda\undefined\newcommand{\Lambda}{\Lambdait}
\DeclareMathSymbol{\Piit}{\mathalpha}{letters}{"05}\let\Pi\undefined\newcommand{\Pi}{\Piit}
\DeclareMathSymbol{\Gammait}{\mathalpha}{letters}{"00}\let\Gamma\undefined\newcommand{\Gamma}{\Gammait}
\DeclareMathSymbol{\Omegait}{\mathalpha}{letters}{"0A}\let\Omega\undefined\newcommand{\Omega}{\Omegait}
\DeclareMathSymbol{\Upsilonit}{\mathalpha}{letters}{"07}\let\Upsilon\undefined\newcommand{\Upsilon}{\Upilonit}
\DeclareMathSymbol{\Thetait}{\mathalpha}{letters}{"02}\let\Theta\undefined\newcommand{\Theta}{\Thetait}

\def\Hom{\mathrm{Hom}}

\def\id{\mathrm{id}}

\def\dim{\mathrm{dim}}
\let\to\undefined\newcommand{\to}{\longrightarrow}
\let\mapsto\undefined\newcommand{\mapsto}{\longmapsto}
\newcommand{\catf}[1]{\mathsf{#1}}
\newcommand{\Proj}{\operatorname{\catf{Proj}}}

\newcommand{\Map}{\catf{Map}}

\newcommand{\IN}{{\normalfont\tiny in}}
\newcommand{\OUT}{{\normalfont\tiny out}}
\usepackage[normalem]{ulem}
\newcommand{\naka}{\catf{N}^\catf{r}}

\newcommand{\ra}[1]{\ \xrightarrow{\ \  #1  \  \ }\ }

\def\Ch{\catf{Ch}_k}

\newcommand{\trr}{\catf{tr}_\catf{r}}

\newcommand{\HS}{\catf{HS}}

\newcommand{\trace}{\catf{t}}

\newcommand{\spr}[1]{\left\langle #1\right\rangle}
\newcommand{\cat}[1]{\mathcal{#1}}

\newcommand{\OC}{\catf{OC}}

\newcommand{\lint}{\int_\mathbb{L}}

\makeatletter
\newcommand{\ostar}{\mathbin{\mathpalette\make@circled\star}}
\newcommand{\make@circled}[2]{%
	\ooalign{$\m@th#1\smallbigcirc{#1}$\cr\hidewidth$\m@th#1#2$\hidewidth\cr}%
}
\newcommand{\smallbigcirc}[1]{%
	\vcenter{\hbox{\scalebox{0.77778}{$\m@th#1\bigcirc$}}}%
}
\makeatother

\definecolor{Blue}  {rgb} {0.282352,0.239215,0.803921}
\definecolor{Green} {rgb} {0.133333,0.545098,0.133333}
\definecolor{Red}   {rgb} {0.803921,0.000000,0.000000}
\definecolor{Violet}{rgb} {0.580392,0.000000,0.827450}
\definecolor{morphismcolor}{rgb}{1,1,1}%\definecolor{morphismcolor}{rgb}{0.39,0.72,0.996}

\setkomafont{caption}{\small\slshape}

\setkomafont{section}{\rmfamily\Large}
\setkomafont{subsection}{\rmfamily}
\setkomafont{subsubsection}{\rmfamily}
\setkomafont{subparagraph}{\rmfamily}
\usepackage{tocloft}

\newtheorem*{theorem*}{Theorem}
\newtheorem*{corollary*}{Corollary}

\begin{document} 
\vspace*{-10mm}
	\begin{flushright}
		\small
		{\sffamily [ZMP-HH/21-3]} \\
		\textsf{Hamburger Beiträge zur Mathematik Nr.~890}\\
		\textsf{CPH-GEOTOP-DNRF151}\\
		\textsf{March 2021}
	\end{flushright}
	
\vspace{8mm}
	
	\begin{center}
		\textbf{\Large{The Trace Field Theory of a Finite Tensor Category}}\\
		\vspace{6mm}

{\large Christoph Schweigert ${}^a$ and  Lukas Woike ${}^b$}

\vspace{3mm}

\normalsize
{\slshape $^a$ Fachbereich Mathematik\\ Universit\"at Hamburg\\
	Bereich Algebra und Zahlentheorie\\
	Bundesstra\ss e 55\\  D-20146 Hamburg }

\vspace*{3mm}	

{\slshape $^b$ Institut for Matematiske Fag\\ K\o benhavns Universitet\\
	Universitetsparken 5 \\  DK-2100 K\o benhavn \O }

	\end{center}
\vspace*{1mm}
	\begin{abstract}\noindent 	
Given a finite tensor category $\cat{C}$, we prove that a modified trace on the tensor ideal of projective objects can be obtained from a suitable trivialization of the Nakayama functor as right $\cat{C}$-module functor. Using a result of Costello, this allows us to associate to any finite tensor category equipped with such a  trivialization of the Nakayama functor a chain complex valued topological conformal field theory, the \emph{trace field theory}. The trace field theory topologically describes the modified trace, the Hattori-Stallings trace, and also the structures induced by them on the Hochschild complex of $\cat{C}$. In this article, we focus on implications in the \emph{linear} (as opposed to differential graded) setting: We use the trace field theory to define a non-unital homotopy commutative product on the Hochschild chains in degree zero. This product is block diagonal and can be described through the handle elements of the trace field theory. Taking the modified trace of the handle elements recovers the Cartan matrix of $\cat{C}$.
\end{abstract}

%\tableofcontents

%\linenumbers % Schaltet Zeilennummerierung ein
%\modulolinenumbers[5]

\normalsize

\section{Introduction and summary}
A \emph{finite category} 
over a fixed field $k$, that we will assume to be algebraically closed throughout,
is an Abelian category enriched over finite-dimensional $k$-vector spaces which has enough projective objects and finitely many isomorphism classes of simple objects; moreover, one requires that every object has finite length. 
Any finite category $\cat{C}$ comes with a right exact endofunctor
\begin{align}
	\naka : \cat{C}\to\cat{C}\ , \quad X \mapsto \naka X:=\int^{Y\in \cat{C}} \cat{C}(X,Y)^* \otimes Y \ , 
\end{align}
the \emph{(right) Nakayama functor}, where $\otimes$ denotes  the tensoring of objects in $\cat{C}$ with finite-dimensional vector spaces.
This Morita invariant description was given in \cite{fss}, and it reduces to the usual definition of the (right) Nakayama functor 
for the category of finite-dimensional modules over a finite-dimensional $k$-algebra.
As a consequence of the coend description of $\naka$,
we obtain in Corollary~\ref{cortheiso}
 natural isomorphisms
	\begin{align} \cat{C}(P,X) \cong \cat{C}(X,\naka P)^* \quad \text{for}
		 \quad X\in\cat{C}\ , \quad  P\in \Proj \cat{C} \ 
\end{align}
turning the subcategory $\Proj\cat{C}\subset \cat{C}$
into an $\naka$-twisted Calabi-Yau category.
 Through the  correspondence between 
 (twisted) Calabi-Yau structures and (twisted) traces, one  obtains the trace
 \begin{align}\trace_P : \cat{C}(P,\naka P) \to k \quad \text{for}\quad P\in \Proj \cat{C} \ . \label{eqngentrace}  \end{align}
 It is now an obvious task to relate this relatively generically constructed
 trace (or rather family of traces) to the 	
modified trace \cite{geerpmturaev,mtrace1,mtrace2,mtrace3,mtrace,bbg18}. Modified traces are
 not only a concept of independent algebraic interest, but   can also 
 be used for the construction of invariants of closed three-dimensional manifolds, see e.g.\ \cite{cgp,bcgpm}.
In such constructions, they serve as a non-semisimple replacement for quantum traces.

In this article, 
we prove in Theorem~\ref{thmmtrace} that, for a finite tensor category in the sense of Etingof-Ostrik \cite{etingofostrik}, i.e.\ a finite category with rigid monoidal product and simple unit,
the trace~\eqref{eqngentrace} on the tensor ideal of projective objects, indeed produces a twisted modified trace.

\begin{reptheorem}{thmmtrace}
	For any finite tensor category $\cat{C}$,
the twisted trace
$	(\trace_P : \cat{C}(P,\naka P) \to k)_{P\in \Proj\cat{C}}$ 
from \eqref{eqngentrace}
is (twisted) cyclic, non-degenerate and satisfies a generalized partial trace property.
Under the additional assumption that on the finite tensor category $\cat{C}$ a pivotal structure has been chosen,
the twisted trace
$	(\trace_P : \cat{C}(P,\naka P) \to k)_{P\in \Proj\cat{C}}$ 
can be naturally identified with a right modified $D$-trace,
where $D\in\cat{C}$ is the distinguished invertible object of $\cat{C}$.
\end{reptheorem}

This uses crucially that by \cite[Theorem~4.26]{fss} the Nakayama functor of a finite tensor category can be expressed as
\begin{align}
	\naka \cong D^{-1}\otimes -^{\vee \vee}
	\end{align}
using the distinguished invertible object $D\in \cat{C}$ \cite{eno-d} and the double dual functor $-^{\vee \vee}$. 
Note that we do not require a pivotal structure to define our traces because the double dual functor can be conveniently
absorbed into the Nakayama functor. If $\cat{C}$ is pivotal, however, we recover the usual definitions.

Our motivation for unraveling the connection between the Nakayama functor and the modified trace
is topological, but does not directly come from invariants of closed three-dimensional manifolds. Instead, we are motivated by two-dimensional topological conformal field theory, a certain type of differential graded two-dimensional open-closed topological field theory:
Suppose that we are given a finite tensor category $\cat{C}$ and 
a \emph{symmetric Frobenius structure}, by which we mean
a certain trivialization of the right Nakayama functor as right $\cat{C}$-module functor relative to a pivotal structure (we give the details in Definition~\ref{defsymfrob}; it will amount to a pivotal structure and a trivialization of the distinguished invertible object). 
Then the trace coming from \emph{this}
particular trivialization of $\naka$ 	
produces, as discussed above,
a Calabi-Yau structure on the tensor ideal $\Proj \cat{C}\subset \cat{C}$.
To this Calabi-Yau structure on $\Proj \cat{C}$,
 Costello's Theorem \cite{costellotcft} associates a 
topological conformal field theory $\Phi_\cat{C}$ 
that we refer to as the \emph{trace field theory} 
of the finite tensor category $\cat{C}$ 
with symmetric Frobenius structure.
On a technical level, $\Phi_\cat{C}:\OC\to\Ch$ is a symmetric monoidal functor from a certain differential graded version of the open-closed bordism category to chain complexes, we recall the details in Section~\ref{sectracefieldtheory}.  

If we evaluate $\Phi_\cat{C}$
 on the \emph{open} part of the two-dimensional bordism category, $\Phi_\cat{C}$ provides topological tools to compute with traces, but only captures information that one could  have  obtained by hand. This is drastically different for the \emph{closed} part of the two-dimensional bordism category: On a closed boundary component, i.e.\ on the circle, we obtain, following again \cite{costellotcft}, the Hochschild complex of $\cat{C}$, i.e.\ the homotopy coend $\lint^{X\in\Proj\cat{C}} \cat{C}(X,X)$ over the endomorphism spaces of projective objects. On this complex, we have an action of the prop provided by the chains on moduli spaces of Riemann surfaces with closed boundary components. Phrased differently,
the trace field theory $\Phi_\cat{C}$
captures the higher structures induced by the modified trace on the Hochschild complex of $\cat{C}$ while, at the same time, being very accessible
through the tools available for computations with Nakayama functors.
These higher structures will be developed in detail elsewhere. For the present article, no homotopy theory is needed, and we focus entirely on the purely linear consequences, i.e.\ on structures induced in homological degree zero. 

\begin{reptheorem}{thmtracefieldtheory}
Let $\cat{C}$ be a finite tensor category with symmetric Frobenius structure and $\Phi_\cat{C}:\OC\to \Ch$ its trace field theory.
	The evaluation of $\Phi_\cat{C}$ on the disk 
	with one incoming open boundary interval whose complementing free boundary carries the label $P\in\Proj\cat{C}$
	\begin{equation} \Phi_\cat{C} \left( \begin{tikzpicture}[scale=0.7]
	\begin{pgfonlayer}{nodelayer}
		\node [style=none] (21) at (0, 2) {};
		\node [style=none] (22) at (0, -1) {};
		\node [style=none] (25) at (-1.5, 1) {};
		\node [style=none] (26) at (-1.5, 0) {};
		\node [style=none] (27) at (2, 0.5) {{\small $P$}};
		\node [style=none] (28) at (1, 2) {};
		\node [style=none] (29) at (-2.25, 0.5) {\IN};
	\end{pgfonlayer}
	\begin{pgfonlayer}{edgelayer}
		\draw [bend left=90, looseness=1.75] (21.center) to (22.center);
		\draw [bend right=90, looseness=1.75] (21.center) to (22.center);
		\draw [style=open, bend right=15, looseness=0.75] (25.center) to (26.center);
	\end{pgfonlayer}
\end{tikzpicture}  \right)\  : \  \cat{C}(P,P)\to k 
	\end{equation}
	is a right modified trace, while the evaluation of $\Phi_\cat{C}$ on the cylinder with one incoming open boundary interval with complementing free boundary label $P\in\Proj\cat{C}$ and one outgoing closed boundary circle
	\begin{align}
		\Phi_\cat{C} \left( \begin{tikzpicture}[scale=0.7]
	\begin{pgfonlayer}{nodelayer}
		\node [style=none] (0) at (2.5, 24) {};
		\node [style=none] (1) at (2.5, 21) {};
		\node [style=none] (2) at (1.5, 23) {};
		\node [style=none] (3) at (1.5, 21.75) {};
		\node [style=none] (4) at (4.25, 22.5) {{\small $P$}};
		\node [style=none] (5) at (6.5, 24) {};
		\node [style=none] (6) at (6.5, 21) {};
		\node [style=none] (7) at (0.5, 22.5) {\IN};
		\node [style=none] (8) at (8, 24) {\OUT};
	\end{pgfonlayer}
	\begin{pgfonlayer}{edgelayer}
		\draw [bend left=90, looseness=1.25] (0.center) to (1.center);
		\draw [bend right=90, looseness=1.25] (0.center) to (1.center);
		\draw [style=open, bend right, looseness=0.75] (2.center) to (3.center);
		\draw [bend left=90, looseness=1.25] (5.center) to (6.center);
		\draw [bend right=90, looseness=1.25] (5.center) to (6.center);
		\draw (0.center) to (5.center);
		\draw (1.center) to (6.center);
	\end{pgfonlayer}
\end{tikzpicture}  \right)\  : \  \cat{C}(P,P)\to \lint^{P\in\Proj\cat{C}} \cat{C}(P,P) 
		\end{align}
	agrees, after taking zeroth homology, with the Hattori-Stallings trace of $\cat{C}$.
	\end{reptheorem}

By evaluation of $\Phi_\cat{C}$
on the pair of pants, we obtain a non-unital multiplication $\star$
on the Hochschild complex
(it will generally not have a unit because the bordism that would normally give us a unit is not admitted in Costello's category $\OC$).
From results of Wahl and Westerland \cite{wahlwesterland}, we can conclude that this multiplication is supported, up to homotopy, in degree zero. Moreover, it is homotopy commutative by construction.
Besides the connection between the Nakayama functor and the modified trace,
the construction of this multiplication or rather its degree zero remnant is one of the main results of this short article and will be one of the key ingredients for future work.
In the present article, we prove that the product $\star$ is block diagonal (Proposition~\ref{propdiag}) and provide a formula for $\star$ (when evaluated on identity morphisms)
involving the \emph{handle elements}
\begin{align}
		\xi_{P,Q} :=    \Phi_\cat{C}\left(    \begin{tikzpicture}
	\begin{pgfonlayer}{nodelayer}
		\node [style=none] (21) at (0, 2) {};
		\node [style=none] (22) at (0, -1) {};
		\node [style=none] (28) at (-2, 0.75) {{\small $P$}};
		\node [style=none] (29) at (0, 1) {};
		\node [style=none] (30) at (0, 0) {};
		\node [style=none] (31) at (0, -0.5) {{\small $Q$}};
		\node [style=none] (32) at (1.5, 1) {};
		\node [style=none] (33) at (1.5, 0) {};
		\node [style=none] (34) at (2.25, 0.5) {\OUT};
	\end{pgfonlayer}
	\begin{pgfonlayer}{edgelayer}
		\draw [bend left=90, looseness=1.75] (21.center) to (22.center);
		\draw [bend right=90, looseness=1.75] (21.center) to (22.center);
		\draw [bend left=90, looseness=1.75] (29.center) to (30.center);
		\draw [bend right=90, looseness=1.75] (29.center) to (30.center);
		\draw [style=open, bend left=15, looseness=0.75] (32.center) to (33.center);
	\end{pgfonlayer}
\end{tikzpicture}    \right)\in\cat{C}(P,P) 
		\quad\text{for}\quad P,Q\in\Proj\cat{C} \ . \label{eqnhandleelementintro}
	\end{align}
of the trace field theory.

\begin{reptheorem}{thmhandleelement}
	Let $\cat{C}$ be a finite tensor category with symmetric Frobenius structure.
	\begin{pnum}

		\item Let $P,Q\in\Proj\cat{C}$.
		Up to boundary in the Hochschild complex of $\cat{C}$,
		the $\star$-product of $\id_P$ and $\id_Q$ is the handle element $\xi_{P,Q}$ of $P$ and $Q$:
		\begin{align} \id_P \star \id_Q \simeq \xi_{P,Q} \ . 
		\end{align}
		
		\item
	All handle elements in the sense of~\eqref{eqnhandleelementintro} are central elements in the endomorphism algebras of $\cat{C}$.

\item 
	The modified trace of the handle element is given by
	\begin{align} \trace_P \xi_{P,Q}=\dim \,\cat{C}(P,Q) \ . \label{eqntraceformulai}
	\end{align}
\end{pnum}
\end{reptheorem}

Formula~\eqref{eqntraceformulai} tells us that the modified trace of the handle elements recovers the entries of the Cartan matrix of $\cat{C}$.
If $P$ is simple, the handle element can be identified with the number
\begin{align}
	\xi_{P,Q} = \frac{\dim\, \cat{C}(P,Q)}{d^\text{m} (P)} \in k \ , 
\end{align}
where $d^\text{m} (P):=\trace_P(\id_P)\in k^\times$ is the modified dimension of $P$.

If we denote for an endomorphism $f:P\to P$ of a projective object $P$
the Hattori-Stallings trace by $\HS(f)\in HH_0(\cat{C})$, we obtain the following statement in homology:

	\begin{repcorollary}{corhs}
	For any finite tensor category $\cat{C}$ with symmetric Frobenius structure,
	\begin{align}
		\trace (\HS (\id_P) \star \HS (\id_Q) ) = \dim\, \cat{C}(P,Q) \quad \text{for}\quad P,Q \in \Proj \cat{C} \ . 
	\end{align}
\end{repcorollary}
	Here we denote the map on $HH_0(\cat{C})$
induced by the modified trace 
again by $\trace$.

 \subparagraph{Conventions.} As already mentioned above, for the entire article, we fix an algebraically closed field $k$ (which is not assumed to have characteristic zero).

Concerning the convention on left and right duality, we follow \cite{egno}:
In a \emph{rigid} monoidal category $\cat{C}$, every object $X \in \cat{C}$ has \begin{itemize}
	\item a \emph{left dual} $X^\vee$
that comes with an evaluation $d_X:X^\vee \otimes X\to I$ and a coevaluation $b_X:I\to X\otimes X^\vee$ subject to the usual zigzag identities,

\item and a \emph{right dual ${^\vee \! X}$} that comes with an evaluation $\widetilde d_X : X\otimes  {^\vee \! X} \to I$ and a coevaluation $\widetilde b_X : I \to {^\vee \! X}\otimes X$ again subject to the usual zigzag identities.
\end{itemize}

 \subparagraph{Acknowledgments.} 
We would like to thank 
Jürgen Fuchs,
Lukas Müller
and Nathalie Wahl
for helpful discussions. 

CS is supported by the Deutsche Forschungsgemeinschaft (DFG, German Research
Foundation) 
under Germany’s Excellence Strategy -- EXC 2121 ``Quantum Universe'' -- 390833306.
LW gratefully acknowledges support by 
the Danish National Research Foundation through the Copenhagen Centre for Geometry
and Topology (DNRF151).

 \subparagraph{Note.} 
While finalizing this manuscript,
the preprint \cite{ss21}
appeared which provides a proof for a connection between the Nakayama functor and modified traces very similar to the one afforded by Theorem~\ref{thmmtrace}.

\section{Traces on finite categories\label{sectracesfin}}
For any finite category $\cat{C}$, the (right) Nakayama functor $\naka : \cat{C}\to\cat{C}$ is given by 
\begin{align} \naka X:= \int^{Y\in \cat{C}} \cat{C}(X,Y)^* \otimes Y  \ .   \label{defeqnnaka}
\end{align}
This is the Morita invariant description
given in \cite{fss} for the usual Nakayama functor 
for finite-dimensional modules over a finite-dimensional algebra $A$
which is given by
\begin{align}\label{eqnnakayamaalg} \naka X= \Hom_A(X,A)^*\cong A^* \otimes_A X  \end{align}
for any finite-dimensional $A$-module $X$.
The Nakayama functor sends projective objects to injective objects.

\begin{proposition}\label{propdualhom}
	For any finite category $\cat{C}$, there is a canonical isomorphism of chain complexes
	\begin{align} \cat{C}(X,Y_\bullet)^* \cong \cat{C}(Y_\bullet, \naka X  ) \end{align} natural in objects $X,Y\in\cat{C}$, where $Y_\bullet$ is a projective resolution of $Y$. 
\end{proposition}

\begin{proof}
	Since every finite category can be written as finite-dimensional modules over a finite-dimensional algebra, we conclude from the comparison of~\eqref{defeqnnaka} and~\eqref{eqnnakayamaalg} that the right hand side of~\eqref{defeqnnaka} can be modeled as a \emph{finite} colimit.
	
	Since $\cat{C}(Y_\bullet,-)$ is exact, the finite colimit used to define $\naka X$ is preserved, which leads to
	\begin{align}
	\cat{C}(Y_\bullet, \naka X  )\cong \int^{Z\in\cat{C}} \cat{C}(Y_\bullet,\cat{C}(X,Z)^*\otimes Z)\cong \int^{Z\in\cat{C}} \cat{C}(Y_\bullet, Z)\otimes \cat{C}(X,Z)^*\cong \cat{C}(X,Y_\bullet)^* \ .
	\end{align}
All coends are computed degree-wise here.
	In the last step, we have used the Yoneda Lemma. 
\end{proof}

The isomorphisms from Proposition~\ref{propdualhom} can be used to obtain a twisted Calabi-Yau structure on $\Proj \cat{C}$. 

\begin{definition}
	An \emph{$(F,G)$-twisted Calabi-Yau category} is a linear category $\cat{A}$ with endofunctors $F,G:\cat{A}\to\cat{A}$ and  isomorphisms $\cat{A}(F(X),Y)\cong \cat{A}(Y,G(X))^*$ natural in $X,Y\in \cat{A}$. 
\end{definition}

In order to avoid overloaded notation, 
we call a twisted Calabi-Yau category \emph{left $F$-twisted} and \emph{right $G$-twisted} if the twist datum $(F,G)$ is given $(F,\id_\cat{A})$ and $(\id_\cat{A},G)$, respectively.
By a \emph{Calabi-Yau category} (without any mention of twists) we will understand an untwisted Calabi-Yau category in the sense that $F=G=\id_\cat{A}$.
A Calabi-Yau category with one object is a symmetric Frobenius algebra.

\begin{corollary}\label{cortheiso}
	For a finite  category $\cat{C}$,
	there are canonical isomorphisms
	\begin{align} \cat{C}(P,X) \cong \cat{C}(X,\naka P)^*  \label{eqntheiso}
	\end{align}
	natural in $X\in\cat{C}$ and $P\in \Proj \cat{C}$.
	In particular, $\Proj \cat{C}$ is a right $\naka$-twisted Calabi-Yau category.
\end{corollary} 

\begin{proof}%[Proof of Corollary~\ref{cortheiso}]
	For $X\in \cat{C}$ and $P\in \Proj \cat{C}$,
	we find  (we denote equivalences of chain complexes aka quasi-isomorphisms by $\simeq$ and isomorphisms, as before, by $\cong$)
	\begin{align}\begin{array}{rclll}
	\cat{C}(X,\naka P)^* &\simeq & \cat{C}(X_\bullet,\naka P)^* && \text{(because $\naka P$ is  injective)} \\ &\cong& \cat{C}(P,X_\bullet)&& \text{(Proposition~\ref{propdualhom})}\\& \simeq &\cat{C}(P,X) && \text{(because $P$ is projective)} \ . \end{array}
	\end{align}
	%	Since, of course, $\cat{C}(X,\naka P)^*$ and $\cat{C}(P,X)$ are concentrated in degree zero, we obtain the isomorphism~\eqref{eqntheiso}.
\end{proof}

\begin{definition}[Trace of a finite category]\label{deftracefc}
	For any finite category $\cat{C}$, we 
	define the pairings
	\begin{align}
	\label{eqntheparings} \spr{  -,-  } : \cat{C}(P,X)\otimes\cat{C}(X,\naka P) \ra{  \eqref{eqntheiso}   } \cat{C}(X,\naka P)^*\otimes \cat{C}(X,\naka P) &\ra{\text{evaluation}} k \\  \quad \text{for}\quad &X\in\cat{C}\ , \ P\in\Proj\cat{C} 
	\end{align}
	and, by considering the case $X=P$ in~\eqref{eqntheparings},
	the  \emph{twisted trace} \begin{align} \trace_P : \cat{C}(P,\naka P) \ra{  \spr{\id_P,-}  } k \quad \text{for}\quad P\in\Proj \cat{C} \label{eqnthetrace} \end{align} on $\cat{C}(P,\naka P)$.
	We refer to the family of maps~\eqref{eqnthetrace},
	where $P$ runs over all projective objects, as the \emph{twisted trace on $\cat{C}$}.
	By an \emph{untwisting} of the twisted trace, we mean a trivialization $\naka \cong \id_\cat{C}$ of $\naka$ (if there exists any) and the resulting identification of the maps \eqref{eqnthetrace} with maps $\cat{C}(P,P)\to k$ that we then refer to as \emph{untwisted trace}, or just \emph{trace} for brevity. 
\end{definition}

	It is important to note that the twisted trace is canonical while the untwisting (if possible) will involve choices.
	An untwisting of the trace is  equivalent to an untwisting of the twisted Calabi-Yau structure from Corollary~\ref{cortheiso}.

The usual correspondence between Calabi-Yau structures and traces can be adapted to the present situation and leads to the following:

\begin{lemma}\label{lemmatrace}
	For any finite  category $\cat{C}$, the twisted trace
	\begin{align}\trace_P : \cat{C}(P,\naka P) \to k \quad \text{for}\quad P\in \Proj \cat{C} \label{eqnthetraces} \end{align} 
	has the following properties:
	\begin{pnum}
		\item Cyclicity: For $P,Q\in\Proj \cat{C}$, $f: P\to \naka Q$ 
		and $g :Q\to P$, we have
		\begin{align}
		\trace_Q(fg)=\trace_P(\naka(g)f) \ . 
		\end{align}
		
		\item Non-degeneracy: The trace is non-degenerate in the sense that the pairings
		\begin{align}
		\cat{C}(P,X)\otimes \cat{C}(X,\naka P) \to k \ , \quad f\otimes g \mapsto \trace_P (gf)  \label{eqnsecondpairing}
		\end{align} 
		are non-degenerate. 
		In fact, they agree with the pairings~\eqref{eqntheparings}.
		
	\end{pnum} 
\end{lemma}

\begin{proof}
	Let $X,Y \in \cat{C}$ and $P,Q\in\Proj\cat{C}$. Naturality of \eqref{eqntheiso} in $X$ means for $a:P\to X$, $b :Y\to \naka P$ and $c: X\to Y$
	\begin{align}
	\spr{  {a},bc}=\spr{  {ca},b} \ . \label{hateqn1}
	\end{align}
	Naturality of \eqref{eqntheiso} in $P$ means for $a:Q\to X$, $b:P\to Q$ and $c:X\to \naka P$
	\begin{align} 
	\spr{  {a},\naka(b)c}=\spr{  ab,c} \ . \label{hateqn2}
	\end{align}
	This implies for $f: P\to \naka Q$ 
	and $g :Q\to P$
	\begin{align}
	\trace_Q(fg)\stackrel{\eqref{eqnthetrace}}{=}\spr{  {\id_Q},fg} \stackrel{\eqref{hateqn1}}{=} \spr{  {g},f} \stackrel{\eqref{hateqn2}}{=} \spr{  {\id_P},\naka(g)f} \stackrel{\eqref{eqnthetrace}}{=} \trace_P(\naka(g)f) \ . 
	\end{align}
	This proves cyclicity.
	Non-degeneracy holds by construction because it follows easily from~\eqref{hateqn1} that the pairing~\eqref{eqnsecondpairing} agrees with $\spr{-,-}$.
\end{proof}

\spaceplease
\section{Traces on finite tensor categories and connection to modified traces\label{tracesonfinitfc}}
We will now turn to finite \emph{tensor} categories  and
connect the construction from
Definition~\ref{deftracefc} to  modified traces \cite{mtrace1,mtrace2,mtrace3,mtrace}.
A (twisted, right) modified trace on the tensor ideal of projective objects in a pivotal finite tensor category is a cyclic, non-degenerate trace that satisfies the right partial trace property as we will discuss in detail below. The first two properties hold very generally for traces constructed from linear trivializations of the Nakayama functor thanks to Lemma~\ref{lemmatrace}.

The partial trace property makes use of the monoidal structure. In order to understand when we can formulate and prove such a property for the trace from Definition~\ref{deftracefc}, one needs to understand the Nakayama functor of a finite tensor category:
Let $\cat{C}$ be any finite tensor category.
We denote by $\cat{C}_\cat{C}$ the finite category $\cat{C}$ as regular right module category over itself and by $\cat{C}^{\vee \vee}$ the finite category $\cat{C}$ as $\cat{C}$-right module with action given by $X.Y:=X\otimes Y^{\vee \vee}$ for $X,Y\in\cat{C}$. 

\begin{theorem}[{\cite[Theorem~3.26]{fss}}] \label{thmnakamodule}
	For any finite tensor category $\cat{C}$, the (right) Nakayama functor is an equivalence $\naka :  \cat{C}_\cat{C} \ra{\simeq} \cat{C}^{\vee \vee}$ of right $\cat{C}$-module categories; in particular, comes with canonical isomorphisms $ \naka(-\otimes X)\cong \naka(-)\otimes X^{\vee\vee}$ for $X\in\cat{C}$. Moreover, $\naka I\cong D^{-1}$, where $D\in\cat{C}$ is the distinguished invertible object and $D^{-1}$ its dual, and hence
	\begin{align} \naka \cong D^{-1}\otimes -^{\vee\vee}
		\end{align}
	by a canonical isomorphism.
\end{theorem}

Together with Proposition~\ref{propdualhom},
this implies:

\begin{corollary}\label{cordualhom}
	For any  finite tensor category $\cat{C}$, there are canonical isomorphisms
	\begin{align} \cat{C}(X,Y_\bullet)^* \cong \cat{C}(Y_\bullet, D^{-1} \otimes  X^{\vee \vee}  ) \end{align} natural in objects $X,Y\in\cat{C}$, where $Y_\bullet$ is a projective resolution of $Y$.
	In particular, any pivotal structure on $\cat{C}$ provides canonical isomorphisms
	\begin{align} \cat{C}(X,Y_\bullet)^* \cong \cat{C}(Y_\bullet, D^{-1} \otimes  X  ) \ .  \end{align}
\end{corollary}

We  now propose a generalization of the partial trace property that does not need a pivotal structure (from our perspective, this will turn out to be more natural):
Let $\cat{C}$ be a finite tensor category.
For $X\in\cat{C}$ and $P\in\Proj\cat{C}$, we may use Theorem~\ref{thmnakamodule} to define a 
map
\begin{align} \cat{C}\left(P\otimes X, \naka(P)\otimes X^{\vee \vee} \right) \to \cat{C}(P, \naka P ) \label{eqnptpart2}
\end{align} sending $f: P\otimes X \to \naka(P)\otimes X^{\vee \vee}$ to
\begin{align}
	P \ra{P\otimes b_X} P \otimes X \otimes X^\vee \ra{f\otimes X^\vee} \naka(P)\otimes X^{\vee \vee} \otimes X^\vee \ra{\naka(P)\otimes d_{X^\vee}} \naka P \ . 
\end{align}

\begin{definition}
	Let $\cat{C}$ be a finite tensor category, $P\in \Proj \cat{C}$ and $X\in\cat{C}$.
	Then we define the 
 \emph{right partial trace} as the composition
\begin{align}
	\trr^X: \cat{C}(P\otimes X, \naka(P\otimes X))  \ra{\text{Theorem~\ref{thmnakamodule}}} \cat{C}(P\otimes X, \naka(P)\otimes X^{\vee \vee} ) \ra{\eqref{eqnptpart2}} \cat{C}(P, \naka(P) ) \ .  \label{eqnpartialtrace0}
\end{align}
\end{definition}

All of this crucially uses that $P\otimes X$ (and also $X\otimes P$) is projective if $P$ is, i.e.\ the ideal property property of $\Proj \cat{C}$.

\begin{remark}\label{rempartialtrace} 
	A pivotal structure is not needed for the definition given here because the double dual is absorbed into the Nakayama functor. In presence of a pivotal structure $\omega : -^{\vee \vee}\cong \id_\cat{C}$, however, our definition specializes to the usual partial trace property for a \emph{right $D$-trace} in the terminology of \cite{mtrace}
	in the sense that the composition
	\begin{align}
		\cat{C}(D\otimes P\otimes X, P\otimes X) \cong %\cat{C}(P\otimes X, D^{-1}\otimes P\otimes X^{\vee\vee}) \cong  
		\cat{C}(P\otimes X, \naka(  P\otimes X)) \ra{\trr^X} \cat{C}(P, \naka(P) ) \ , 
	\end{align}
where the first isomorphism uses duality, Theorem~\ref{thmnakamodule} and the (inverse) pivotal structure,
	is the usual partial trace.
\end{remark}

\begin{proposition}\label{proppartialtrace}
	For any finite tensor category $\cat{C}$,
	the twisted trace 
	\begin{align}\trace_P : \cat{C}(P,\naka P) \to k \quad \text{for}\quad P\in \Proj \cat{C}  \end{align} 
	from Definition~\ref{deftracefc}
	satisfies the right partial trace property: For $X\in\cat{C}$ and $P\in\Proj \cat{C}$ and any morphism $f:P\otimes X\to \naka(P\otimes X)$ 
	\begin{align} \trace_P \trr^X (f)=\trace_{P\otimes X}( f) \ .\label{eqnpartialtrace} \end{align}
\end{proposition}

\begin{proof}
	For $X,Y\in\cat{C}$, $P\in\Proj\cat{C}$ and a projective resolution $Y_\bullet$ of $Y$,
	consider the following diagram in which all maps are isomorphisms (we explain all parts of the diagram and its commutativity afterwards):\small
	\begin{equation}
		\begin{tikzcd}
			\cat{C}(P\otimes X,Y_\bullet)^*\ar[d,swap,"\vee"] & \ar[l,swap,"\text{YL}"] \ar[d,swap,"\vee"]\ar[rr,"(\diamond)"] \int^{Z\in\cat{C}} \cat{C}(Y_\bullet,Z) \otimes \cat{C}(P\otimes X,Z)^* && \cat{C}(Y_\bullet,\naka (P\otimes X))\ar[dd,swap,"\naka(-\otimes X)\cong \naka(-)\otimes X^{\vee\vee}"] \\ \cat{C}(P,Y_\bullet \otimes X^\vee)^* & \ar[l,swap,"\text{YL}"] \int^{Z\in\cat{C}} \cat{C}(Y_\bullet,Z)\otimes\cat{C}(P,Z\otimes X^\vee)^* \ar[d,"\text{relabeling}"] \\ & \ar[lu,"\text{YL}"]\int^{Z' \in \cat{C}} \cat{C}(Y_\bullet\otimes X^\vee,Z')\otimes \cat{C}(P,Z')^* \ar[r,"(\diamond)"] & \cat{C}(Y_\bullet \otimes X^\vee,\naka P )  \ar[r,"\vee"]  & \cat{C}(Y_\bullet,\naka P \otimes X^{\vee \vee}) \ . 
		\end{tikzcd}
	\end{equation}\normalsize
	The isomorphisms labeled `YL' and `$\vee$' come from the Yoneda Lemma and duality, respectively. 
	The isomorphisms  $(\diamond)$
	pull  the coend and the tensoring with vector spaces out of the hom functor using  exactness of $\cat{C}(Y_\bullet,-)$
	 (they follow essentially from the definition~\eqref{defeqnnaka} of the Nakayama functor).
	The `relabeling' isomorphism  $\int^{Z\in\cat{C}} \cat{C}(Y_\bullet,Z)\otimes\cat{C}(P,Z\otimes X^\vee)^*\to \int^{Z' \in \cat{C}} \cat{C}(Y_\bullet\otimes X^\vee,Z')\otimes \cat{C}(P,Z')^*$ 
	sends 
	$
	f \otimes \alpha \in \cat{C}(Y_\bullet,Z)\otimes\cat{C}(P,Z\otimes X^\vee)^*
	$
	living in the summand indexed by $Z$ of the coend $\int^{Z\in\cat{C}} \cat{C}(Y_\bullet,Z)\otimes\cat{C}(P,Z\otimes X^\vee)^*$
	to $(f\otimes X^\vee) \otimes \alpha$ living in the summand indexed by $Z\otimes X^\vee$ of the coend $\int^{Z' \in \cat{C}} \cat{C}(Y_\bullet\otimes X^\vee,Z')\otimes \cat{C}(P,Z')^*$.
	The vertical isomorphism on the very right uses Theorem~\ref{thmnakamodule}. 
	
	In fact, the isomorphism  $\naka(-\otimes X)\cong \naka(-)\otimes X^{\vee\vee}$ can be \emph{obtained} by extracting the isomorphism $\cat{C}(Y_\bullet,\naka (P\otimes X))\to \cat{C}(Y_\bullet,\naka P\otimes X^{\vee\vee})$ by going in counterclockwise direction in the hexagon on the right (this follows from an analysis of the proof of \cite[Theorem~3.18]{fss}). 
	As a consequence, the hexagon on the right commutes.  
	
	A direct computation shows that the square and the triangle on the left commute. Therefore, the entire diagram commutes.
	
	After taking the linear dual of the entire diagram and remembering that the isomorphisms `YL' and $(\diamond)$ combine into the isomorphisms from Proposition~\ref{propdualhom}, we see that the diagram
	\begin{equation}
		\begin{tikzcd}
			\ar[rrr,"\text{Proposition~\ref{propdualhom}}"] \ar[dd,swap,"\vee"] \cat{C}(P\otimes X,Y_\bullet) &&& \cat{C}(Y_\bullet,\naka( P\otimes X)) ^*\cong \cat{C}(Y_\bullet,\naka P\otimes X^{\vee\vee}) ^*  \ar[dd,"\vee"]  \\ \\
			\cat{C}(P,Y_\bullet\otimes X^\vee) 	 \ar[rrr,swap,"\text{Proposition~\ref{propdualhom}}"]    &&& \cat{C}(Y_\bullet\otimes X^\vee,\naka P) ^*  \\
		\end{tikzcd}
	\end{equation}
	commutes. 
	Since $P$ is projective (and hence $\naka P$ injective --- in fact, the projective objects in $\cat{C}$  even coincide with the injective ones), this reduces to the commutative diagram
	\begin{equation}
		\begin{tikzcd}
			\ar[rrr,"\text{Corollary~\ref{cortheiso}}"] \ar[dd,swap,"\vee"] \cat{C}(P\otimes X,Y) &&& \cat{C}(Y,\naka( P\otimes X)) ^*\cong \cat{C}(Y,\naka P\otimes X^{\vee\vee}) ^*  \ar[dd,"\vee"]  \\ \\
			\cat{C}(P,Y\otimes X^\vee) 	 \ar[rrr,swap,"\text{Corollary~\ref{cortheiso}}"]    &&& \cat{C}(Y\otimes X^\vee,\naka P) ^*  \\
		\end{tikzcd}
	\end{equation}
	in which the horizontal maps have specialized to the ones from Corollary~\ref{cortheiso}.
	If we spell  out the commutativity of this diagram 
	in equations for morphisms $g:P\otimes X\to Y$ and $h:Y\otimes X^\vee \to \naka P$,
	we obtain with the bracket notation from Definition~\ref{deftracefc}
(we use here additionally the	graphical calculus for morphisms in a monoidal category --- to be read from bottom to top; we refer to \cite{kassel} for a textbook treatment)
	
	\begin{equation}{\footnotesize\begin{tikzpicture}[scale=0.7]
	\begin{pgfonlayer}{nodelayer}
		\node [style=none] (2) at (-10, 4) {};
		\node [style=none] (3) at (-11, 0) {};
		\node [style=none] (4) at (-9, 0) {};
		\node [style=none] (5) at (-11, -0.75) {$P$};
		\node [style=none] (6) at (-9, -0.75) {$X$};
		\node [style=none] (7) at (-10, 4.75) {$Y$};
		\node [style=none] (8) at (-5.75, 1.5) {};
		\node [style=none] (9) at (-3.75, 1.5) {};
		\node [style=none] (10) at (-3.75, 4) {};
		\node [style=none] (11) at (-6, 4) {};
		\node [style=none] (12) at (-10.5, 1.75) {};
		\node [style=none] (13) at (-9.5, 1.75) {};
		\node [style=none] (14) at (-7, 0) {};
		\node [style=none] (15) at (-6, 4.75) {$\naka P$};
		\node [style=none] (16) at (-3.75, 4.75) {$X^{\vee \vee}$};
		\node [style=none] (17) at (-7, -0.75) {$Y$};
		\node [style=none] (18) at (-8, 2) {$,$};
		\node [style=none] (19) at (-13, 2) {};
		\node [style=none] (20) at (-12, 5) {};
		\node [style=none] (21) at (-12, -1) {};
		\node [style=none] (22) at (-1.25, 2) {};
		\node [style=none] (23) at (-2.25, 5) {};
		\node [style=none] (24) at (-2.25, -1) {};
		\node [style=none] (27) at (4, 4) {};
		\node [style=none] (28) at (3, 0) {};
		\node [style=none] (30) at (3, -0.75) {$P$};
		\node [style=none] (31) at (6.25, 4.75) {$X^\vee$};
		\node [style=none] (32) at (4, 4.75) {$Y$};
		\node [style=none] (33) at (4.25, 1.75) {};
		\node [style=none] (34) at (6.25, 1.75) {};
		\node [style=none] (35) at (6.25, 4.25) {};
		\node [style=none] (37) at (3.5, 1.75) {};
		\node [style=none] (43) at (7, 2) {$,$};
		\node [style=none] (44) at (1.25, 2) {};
		\node [style=none] (45) at (2.25, 5) {};
		\node [style=none] (46) at (2.25, -1) {};
		\node [style=none] (47) at (11.75, 2) {};
		\node [style=none] (48) at (10.75, 5) {};
		\node [style=none] (49) at (10.75, -1) {};
		\node [style=none] (50) at (0, 2) {$=$};
		\node [style=none] (54) at (8.75, 4) {};
		\node [style=none] (55) at (7.75, 0) {};
		\node [style=none] (56) at (9.75, 0) {};
		\node [style=none] (57) at (7.75, -0.75) {$Y$};
		\node [style=none] (58) at (9.75, -0.75) {$X^\vee$};
		\node [style=none] (59) at (8.75, 4.75) {$\naka P$};
		\node [style=none] (60) at (8.25, 1.75) {};
		\node [style=none] (61) at (9.25, 1.75) {};
		\node [style=bigbox] (62) at (-10, 2) {$g$};
		\node [style=bigbox] (63) at (-6, 2) {$h$};
		\node [style=bigbox] (64) at (4, 2) {$g$};
		\node [style=bigbox] (65) at (8.75, 2) {$h$};
		\node [style=none] (66) at (13, 2) {$.$};
	\end{pgfonlayer}
	\begin{pgfonlayer}{edgelayer}
		\draw [in=270, out=-90, looseness=1.50] (8.center) to (9.center);
		\draw (3.center) to (12.center);
		\draw (4.center) to (13.center);
		\draw (9.center) to (10.center);
		\draw [style=thick] (19.center) to (20.center);
		\draw [style=thick] (19.center) to (21.center);
		\draw [style=thick] (22.center) to (24.center);
		\draw [style=thick] (22.center) to (23.center);
		\draw [in=270, out=-90, looseness=1.50] (33.center) to (34.center);
		\draw (28.center) to (37.center);
		\draw (34.center) to (35.center);
		\draw [style=thick] (44.center) to (45.center);
		\draw [style=thick] (44.center) to (46.center);
		\draw [style=thick] (47.center) to (49.center);
		\draw [style=thick] (47.center) to (48.center);
		\draw (55.center) to (60.center);
		\draw (56.center) to (61.center);
		\draw (62) to (2.center);
		\draw (14.center) to (63);
		\draw (63) to (11.center);
		\draw (64) to (27.center);
		\draw (65) to (54.center);
	\end{pgfonlayer}
\end{tikzpicture}} \label{eqntrace}
	\end{equation}	
	As another preparation, recall that the double dual functor $-^{\vee \vee}:\cat{C}\to\cat{C}$ is monoidal, hence it  preserves the duality pairing $d _{{^\vee X}} : X \otimes {^\vee X} \to I$ (we use here the canonical identification ${^\vee (X^\vee)}\cong X$) and therefore sends it to $d_{X^\vee} :X^{\vee \vee} \otimes X^\vee \to I$. Using Theorem~\ref{thmnakamodule}   we find 
	the equality
	of morphisms
	\begin{align}	\naka (P\otimes d _{{^\vee X}})=\naka P\otimes d_{X^\vee} : \naka (P) \otimes X^{\vee\vee} \otimes X^\vee \to \naka P \ , \end{align}
	which implies for a morphism $f:P\otimes X\to \naka (P\otimes X)$
	
	\begin{equation}{\footnotesize\begin{tikzpicture}
	\begin{pgfonlayer}{nodelayer}
		\node [style=none] (2) at (-10.5, 6) {};
		\node [style=none] (3) at (-11, 0) {};
		\node [style=none] (5) at (-11, -0.75) {$P$};
		\node [style=none] (6) at (-9.75, 0.75) {$X$};
		\node [style=none] (7) at (-8.75, 3) {$X^{\vee \vee}$};
		\node [style=none] (8) at (-9.5, 2) {};
		\node [style=none] (9) at (-7.75, 2) {};
		\node [style=none] (10) at (-7.75, 4) {};
		\node [style=none] (12) at (-10.5, 2) {};
		\node [style=none] (15) at (-10.5, 6.75) {$\naka P$};
		\node [style=none] (62) at (-9.5, 4) {};
		\node [style=none] (63) at (-12, 2) {$=$};
		\node [style=none] (64) at (-13.25, 2) {{\normalsize $\trr^X f$}};
		\node [style=none] (66) at (-4.5, 4) {};
		\node [style=none] (67) at (-5, 0) {};
		\node [style=none] (68) at (-5, -0.75) {$P$};
		\node [style=none] (69) at (-3.75, 0.5) {$X$};
		\node [style=none] (70) at (-2.75, 3) {$X^{\vee \vee}$};
		\node [style=none] (71) at (-3.5, 2) {};
		\node [style=none] (72) at (-1.75, 2) {};
		\node [style=none] (73) at (-1.75, 4) {};
		\node [style=none] (74) at (-4.5, 2) {};
		\node [style=none] (75) at (-4.5, 6.75) {$\naka P$};
		\node [style=none] (77) at (-6.25, 2) {$=$};
		\node [style=bigbox] (79) at (-3.5, 4) {$\naka(P\otimes d_{{^\vee X}})$};
		\node [style=none] (80) at (-4.5, 6) {};
		\node [style=none] (81) at (-1, 3) {$X^\vee$};
		\node [style=none] (82) at (-7, 3) {$X^\vee$};
		\node [style=none] (83) at (0, 2) {$.$};
		\node [style=bigbox] (84) at (-10, 2) {$f$};
		\node [style=bigbox] (85) at (-4, 2) {$f$};
	\end{pgfonlayer}
	\begin{pgfonlayer}{edgelayer}
		\draw [bend right=90, looseness=3.50] (8.center) to (9.center);
		\draw (3.center) to (12.center);
		\draw (9.center) to (10.center);
		\draw (8.center) to (62.center);
		\draw [bend left=90, looseness=2.50] (62.center) to (10.center);
		\draw (12.center) to (2.center);
		\draw [bend right=90, looseness=3.50] (71.center) to (72.center);
		\draw (67.center) to (74.center);
		\draw (72.center) to (73.center);
		\draw (71.center) to (79);
		\draw (74.center) to (80.center);
	\end{pgfonlayer}
\end{tikzpicture}}\label{trace2}
	\end{equation}
	\normalsize
	The desired equality~\eqref{eqnpartialtrace} now follows from:	
	\footnotesize
	\begin{equation}\begin{tikzpicture}[scale=0.9]
	\begin{pgfonlayer}{nodelayer}
		\node [style=none] (63) at (-15, -5) {$=$};
		\node [style=none] (64) at (-17, -5) {\normalsize$\trace_P \trr^X f$\footnotesize};
		\node [style=none] (66) at (-8.5, -4) {};
		\node [style=none] (67) at (-9, -8) {};
		\node [style=none] (68) at (-9, -8.75) {$P$};
		\node [style=none] (69) at (-7.75, -7.5) {$X$};
		\node [style=none] (70) at (-6.75, -5) {$X^{\vee \vee}$};
		\node [style=none] (71) at (-7.5, -6) {};
		\node [style=none] (72) at (-5.75, -6) {};
		\node [style=none] (73) at (-5.75, -4) {};
		\node [style=none] (74) at (-8.5, -6) {};
		\node [style=none] (75) at (-8.5, -1.25) {$\naka P$};
		\node [style=none] (80) at (-8.5, -2) {};
		\node [style=none] (81) at (-5, -5) {$X^\vee$};
		\node [style=none] (88) at (-11, -8.75) {$P$};
		\node [style=none] (101) at (-10, -5) {$,$};
		\node [style=none] (102) at (-13, -5) {};
		\node [style=none] (103) at (-12, -1) {};
		\node [style=none] (104) at (-12, -9) {};
		\node [style=none] (125) at (-1, -5) {$=$};
		\node [style=none] (135) at (-11, -8) {};
		\node [style=none] (136) at (-11, -2) {};
		\node [style=bigbox] (137) at (-8, -6) {$f$};
		\node [style=bigbox] (138) at (-7.5, -4) {$\naka(P\otimes d_{{^\vee X}})$};
		\node [style=none] (139) at (-4, -1) {};
		\node [style=none] (140) at (-4, -9) {};
		\node [style=none] (141) at (-3, -5) {};
		\node [style=none] (143) at (7, -8) {};
		\node [style=none] (144) at (7, -8.75) {$P$};
		\node [style=none] (145) at (8.25, -6.5) {$X$};
		\node [style=none] (146) at (8.5, -1.25) {$X^{\vee \vee}$};
		\node [style=none] (147) at (8.5, -5) {};
		\node [style=none] (148) at (10, -5) {};
		\node [style=none] (149) at (10, -2) {};
		\node [style=none] (150) at (7.5, -5) {};
		\node [style=none] (151) at (6.75, -1.25) {$\naka P$};
		\node [style=none] (152) at (7, -2) {};
		\node [style=none] (153) at (10.25, -1.25) {$X^\vee$};
		\node [style=none] (154) at (3, -8.75) {$P$};
		\node [style=none] (155) at (6, -5) {$,$};
		\node [style=none] (156) at (1, -5) {};
		\node [style=none] (157) at (2, -1) {};
		\node [style=none] (158) at (2, -9) {};
		\node [style=none] (160) at (3, -8) {};
		\node [style=none] (161) at (3, -2) {};
		\node [style=bigbox] (162) at (8, -5) {$f$};
		\node [style=none] (164) at (11.25, -1) {};
		\node [style=none] (165) at (11.25, -9) {};
		\node [style=none] (166) at (12.25, -5) {};
		\node [style=none] (167) at (4, -4) {};
		\node [style=none] (168) at (5, -4) {};
		\node [style=none] (169) at (4, -8) {};
		\node [style=none] (170) at (5, -8) {};
		\node [style=none] (171) at (4, -8.75) {$X$};
		\node [style=none] (172) at (5, -8.75) {${^\vee X}$};
		\node [style=none] (173) at (8.5, -2) {};
		\node [style=none] (174) at (-15, -15) {$=$};
		\node [style=none] (176) at (-6.75, -18) {};
		\node [style=none] (177) at (-6.75, -18.75) {$P$};
		\node [style=none] (179) at (-4.75, -11.25) {$X^{\vee \vee}$};
		\node [style=none] (180) at (-5.25, -15) {};
		\node [style=none] (181) at (-4.75, -18) {};
		\node [style=none] (183) at (-6.25, -15) {};
		\node [style=none] (184) at (-6.75, -11.25) {$\naka P$};
		\node [style=none] (185) at (-6.75, -12) {};
		\node [style=none] (186) at (-4.75, -18.75) {$X$};
		\node [style=none] (187) at (-11.75, -18.75) {$P$};
		\node [style=none] (188) at (-7.75, -15) {$,$};
		\node [style=none] (189) at (-13.75, -15) {};
		\node [style=none] (190) at (-12.75, -11) {};
		\node [style=none] (191) at (-12.75, -19) {};
		\node [style=none] (192) at (-11.75, -18) {};
		\node [style=none] (193) at (-11.75, -12) {};
		\node [style=bigbox] (194) at (-5.75, -15) {$f$};
		\node [style=none] (195) at (-3.5, -11) {};
		\node [style=none] (196) at (-3.5, -19) {};
		\node [style=none] (197) at (-2.5, -15) {};
		\node [style=none] (204) at (-4.75, -12) {};
		\node [style=none] (205) at (-10.75, -14) {};
		\node [style=none] (206) at (-9.75, -14) {};
		\node [style=none] (207) at (-10.75, -18) {};
		\node [style=none] (208) at (-9.75, -16) {};
		\node [style=none] (209) at (-10.75, -18.75) {$X$};
		\node [style=none] (210) at (-8.75, -11.25) {$X$};
		\node [style=none] (211) at (-8.75, -16) {};
		\node [style=none] (212) at (-8.75, -12) {};
		\node [style=none] (213) at (-1, -15) {$=$};
		\node [style=none] (215) at (6, -18) {};
		\node [style=none] (216) at (6, -18.75) {$P$};
		\node [style=none] (218) at (8, -11.25) {$X^{\vee \vee}$};
		\node [style=none] (219) at (7.5, -15) {};
		\node [style=none] (220) at (8, -18) {};
		\node [style=none] (221) at (6.5, -15) {};
		\node [style=none] (222) at (6, -11.25) {$\naka P$};
		\node [style=none] (223) at (6, -12) {};
		\node [style=none] (224) at (8, -18.75) {$X$};
		\node [style=none] (225) at (3, -18.75) {$P$};
		\node [style=none] (226) at (5, -15) {$,$};
		\node [style=none] (227) at (1, -15) {};
		\node [style=none] (228) at (2, -11) {};
		\node [style=none] (229) at (2, -19) {};
		\node [style=none] (230) at (3, -18) {};
		\node [style=none] (231) at (3, -12) {};
		\node [style=bigbox] (232) at (7, -15) {$f$};
		\node [style=none] (233) at (9.25, -11) {};
		\node [style=none] (234) at (9.25, -19) {};
		\node [style=none] (235) at (10.25, -15) {};
		\node [style=none] (236) at (8, -12) {};
		\node [style=none] (239) at (4, -18) {};
		\node [style=none] (241) at (4, -18.75) {$X$};
		\node [style=none] (244) at (4, -12) {};
		\node [style=none] (245) at (12, -15) {$=$};
		\node [style=none] (246) at (14, -15) {\normalsize$\trace_{P\otimes X}(f)$\footnotesize};
		\node [style=none] (247) at (-15, -4) {\eqref{trace2}};
		\node [style=none] (248) at (-1, -4) {\eqref{hateqn2}};
		\node [style=none] (249) at (-15, -14) {\eqref{eqntrace}};
	\end{pgfonlayer}
	\begin{pgfonlayer}{edgelayer}
		\draw [bend right=90, looseness=3.50] (71.center) to (72.center);
		\draw (67.center) to (74.center);
		\draw (72.center) to (73.center);
		\draw (74.center) to (80.center);
		\draw [style=thick] (102.center) to (103.center);
		\draw [style=thick] (102.center) to (104.center);
		\draw (135.center) to (136.center);
		\draw (71.center) to (138);
		\draw [style=thick] (139.center) to (141.center);
		\draw [style=thick] (140.center) to (141.center);
		\draw [bend right=90, looseness=4.00] (147.center) to (148.center);
		\draw (143.center) to (150.center);
		\draw (148.center) to (149.center);
		\draw (150.center) to (152.center);
		\draw [style=thick] (156.center) to (157.center);
		\draw [style=thick] (156.center) to (158.center);
		\draw (160.center) to (161.center);
		\draw [style=thick] (164.center) to (166.center);
		\draw [style=thick] (165.center) to (166.center);
		\draw (167.center) to (169.center);
		\draw (168.center) to (170.center);
		\draw [bend left=90, looseness=3.25] (167.center) to (168.center);
		\draw (147.center) to (173.center);
		\draw (180.center) to (181.center);
		\draw (176.center) to (183.center);
		\draw (183.center) to (185.center);
		\draw [style=thick] (189.center) to (190.center);
		\draw [style=thick] (189.center) to (191.center);
		\draw (192.center) to (193.center);
		\draw [style=thick] (195.center) to (197.center);
		\draw [style=thick] (196.center) to (197.center);
		\draw (180.center) to (204.center);
		\draw (205.center) to (207.center);
		\draw (206.center) to (208.center);
		\draw [bend left=90, looseness=3.25] (205.center) to (206.center);
		\draw [bend right=90, looseness=3.25] (208.center) to (211.center);
		\draw (211.center) to (212.center);
		\draw (219.center) to (220.center);
		\draw (215.center) to (221.center);
		\draw (221.center) to (223.center);
		\draw [style=thick] (227.center) to (228.center);
		\draw [style=thick] (227.center) to (229.center);
		\draw (230.center) to (231.center);
		\draw [style=thick] (233.center) to (235.center);
		\draw [style=thick] (234.center) to (235.center);
		\draw (219.center) to (236.center);
		\draw (239.center) to (244.center);
	\end{pgfonlayer}
\end{tikzpicture}
	\end{equation}
	\normalsize\end{proof}

\begin{theorem}\label{thmmtrace}
	For any finite tensor category $\cat{C}$,
	the twisted trace
	$	(\trace_P : \cat{C}(P,\naka P) \to k)_{P\in \Proj\cat{C}}$ 
	from Definition~\ref{deftracefc}
	is cyclic, non-degenerate and satisfies the partial trace property
	in the sense of Proposition~\ref{proppartialtrace}.
	Under the additional assumption that on the finite tensor category $\cat{C}$ a pivotal structure has been chosen,
	the twisted trace
	$(	\trace_P : \cat{C}(P,\naka P) \to k)_{P\in\Proj\cat{C}}$ 
	from Definition~\ref{deftracefc}
	can be naturally identified with a right modified $D$-trace,
	where $D\in\cat{C}$ is the distinguished invertible object of $\cat{C}$.
\end{theorem}

\begin{remark}
	More precisely, the twisted trace from Definition~\ref{deftracefc}
	yields a \emph{canonical} right modified $D$-trace and thereby 
	trivializes the $k^\times$-torsor of right modified $D$-traces in a canonical way. 
\end{remark}

\begin{proof}[{\slshape Proof of Theorem~\ref{thmmtrace}}]
	We use the pivotal structure $\omega:-^{\vee\vee}\cong \id_\cat{C}$ to obtain isomorphisms \begin{align} \cat{C}(P,\naka P)\stackrel{\text{Theorem~\ref{thmnakamodule}}}{\cong} \cat{C}(P,D^{-1}\otimes P^{\vee \vee}) \stackrel{\omega \ \text{and duality}}{\cong} \cat{C}(D\otimes P,P) \quad \text{for}\quad P\in\Proj \cat{C}  \ .  \label{eqnthemaps}
	\end{align} 
	As a consequence, the twisted trace from Definition~\ref{deftracefc} gives us maps
	$\cat{C}(D\otimes P,P)\to k$ which are cyclic and non-degenerate (Lemma~\ref{lemmatrace}).
	Moreover, Proposition~\ref{proppartialtrace} combined with Remark~\ref{rempartialtrace} gives us the usual partial trace property in presence of a pivotal structure.
	
	Note that one needs really a \emph{monoidal}
	 isomorphism $-^{\vee\vee}\cong \id_\cat{C}$ 
	to get the desired maps $\cat{C}(D\otimes P,P)\to k$. If $\omega$ is just linear, one would get maps $\cat{C}(D\otimes P,P)\to k$, but  they would not necessarily satisfy the partial trace property:
	The proof of the partial trace property in its $\naka$-twisted version (Proposition~\ref{proppartialtrace}) relies on the monoidal structure of $-^{\vee\vee}$. The partial trace property only transfers along the isomorphisms $\cat{C}(P,D^{-1}\otimes P^{\vee \vee})\cong \cat{C}(P,D^{-1}\otimes P)$ if $-^{\vee\vee}$ is replaced by $\id_\cat{C}$ \emph{as monoidal functor}.
\end{proof}

\spaceplease
\section{The trace field theory\label{sectracefieldtheory}}
We  now introduce the topological conformal field theory induced by the Calabi-Yau structure appearing the previous section.
To this end, let us recall from \cite{costellotcft} the definition of the (differential graded) open-closed two-dimensional cobordism category, see \cite{egas,wahlwesterland} for models of this symmetric monoidal differential graded category in terms of fat graphs.
An \emph{open-closed Riemann surface}
is a Riemann surface with the following data: 
\begin{itemize}
	
	\item A subset of its boundary components, the so-called \emph{closed boundary components}. They are para\-metrized and labeled as incoming or outgoing.
	
	\item A finite number of embedded intervals in the remaining boundary components,
	the so-called \emph{open boundary intervals}.
	They are also parametrized and labeled as incoming or outgoing.
	
\end{itemize} 
The free boundary components are defined as the complement (in the boundary) of the closed boundary components and the open boundary intervals. 
It will be required  that each connected component of the Riemann surface has at least one free boundary component or at least one incoming closed boundary.
An example (that additionally contains certain labels that will be discussed in a moment) is depicted in Figure~\ref{figoc}.

One can now define
the symmetric monoidal differential graded category $\OC$ of \emph{open-closed cobordisms} for a set $\Lambda$ of labels (that we will fix later and that will be suppressed in the notation; the set of labels is sometimes referred to as set of `D-branes'):
The objects are pairs of finite sets $O$ and $C$ (that in a moment will play the role of open boundary intervals and closed boundary components of Riemann surfaces) and two maps $s,t : O\to \Lambda$ (that attach a `start' and an `end' label 
to any
open boundary). 
The chain complex of morphisms from $(O,C,s,t)$ to $(O',C',s',t')$ is given by the $k$-chains on the moduli space of Riemann surfaces $\Sigma$ 
with \begin{itemize}
	\item an identification of its set of incoming open and incoming closed boundary components with $(O,C)$, an identification of its set of outgoing open and outgoing closed boundary components with $(O',C')$,
	\item a label in the set $\Lambda$ of D-branes for each free boundary component 
\end{itemize} 
subject to the following requirement:
First observe that any incoming open boundary interval $o\in O$ inherits a label for its start point and its end point, namely the label of the free boundary component that it is bounded by. We require that this label agrees with $(s(o),t(o))$; the analogous requirement is imposed for outgoing open boundary intervals.
Explicitly, for the objects $X=(O,C,s,t)$ and 
$X'=(O',C',s',t')$,
the morphism complexes are given, up to equivalence, by
\begin{align}
	\OC \left( X,X'  \right) \simeq \bigoplus_{S : X \to X'} C_*(B \Map(S);k) \ , 
	\end{align}
where the direct sum is running over all topological types of compact oriented open-closed bordisms $S$ with incoming and outgoing boundary described by $X$ and $X'$, respectively, and $\Map(S)$ is the mapping class group of $S$; we refer to \cite{egas,wahlwesterland} for a description of these morphism complexes by means of classifying spaces of categories of fat graphs.
Composition in $\OC$ is by gluing. Disjoint union provides a symmetric monoidal structure.

\begin{figure}[h]\centering
	\begin{tikzpicture}[scale=0.6]
	\begin{pgfonlayer}{nodelayer}
		\node [style=none] (0) at (-11, 5) {};
		\node [style=none] (1) at (-11, 2) {};
		\node [style=none] (2) at (-11, 0) {};
		\node [style=none] (3) at (-11, -3) {};
		\node [style=none] (4) at (10, 2.5) {};
		\node [style=none] (5) at (10, -0.5) {};
		\node [style=none] (6) at (-4.5, 1) {};
		\node [style=none] (7) at (-2.5, 1) {};
		\node [style=none] (8) at (-5.5, 1.25) {};
		\node [style=none] (9) at (-1.5, 1.25) {};
		\node [style=none] (10) at (2.5, 1.5) {};
		\node [style=none] (11) at (2.5, 0) {};
		\node [style=none] (12) at (-12, -1) {};
		\node [style=none] (13) at (-12, -2) {};
		\node [style=none] (14) at (-10, -1) {};
		\node [style=none] (15) at (-10, -2) {};
		\node [style=none] (16) at (11, 1.75) {};
		\node [style=none] (17) at (11, 0.25) {};
		\node [style=none] (18) at (-11, 0.5) {{\small $P$}};
		\node [style=none] (19) at (-11, -3.75) {{\small $Q$}};
		\node [style=none] (20) at (8.25, 1) {{\small $R$}};
		\node [style=none] (21) at (2.5, 2) {{\small $S$}};
		\node [style=none] (22) at (-16.5, 5) {{\small $c$}};
		\node [style=none] (23) at (-16, 5) {};
		\node [style=none] (24) at (-12.5, 4) {};
		\node [style=none] (25) at (-12.75, -1.5) {};
		\node [style=none] (26) at (-15, -1.5) {};
		\node [style=none] (27) at (-6, -5) {};
		\node [style=none] (28) at (-9.5, -1.5) {};
		\node [style=none] (29) at (-16, -1.5) {{\small $o_1$}};
		\node [style=none] (30) at (-5.5, -5.25) {{\small $o_2$}};
		\node [style=none] (31) at (11.5, 1) {};
		\node [style=none] (32) at (14, 3) {};
		\node [style=none] (33) at (14.5, 3.25) {{\small $o'$}};
		\node [style=none] (34) at (-12, 5) {\IN};
		\node [style=none] (35) at (-12.5, -2.75) {\IN};
		\node [style=none] (36) at (10.75, 3) {\OUT};
	\end{pgfonlayer}
	\begin{pgfonlayer}{edgelayer}
		\draw [bend left=90, looseness=1.25] (0.center) to (1.center);
		\draw [bend right=90, looseness=1.25] (0.center) to (1.center);
		\draw [bend left=90, looseness=1.25] (2.center) to (3.center);
		\draw [bend right=90, looseness=1.25] (2.center) to (3.center);
		\draw [bend left=90, looseness=1.25] (4.center) to (5.center);
		\draw [bend right=90, looseness=1.25] (4.center) to (5.center);
		\draw [bend right=90, looseness=2.75] (2.center) to (1.center);
		\draw [in=180, out=0, looseness=0.50] (0.center) to (4.center);
		\draw [in=-180, out=0, looseness=0.75] (3.center) to (5.center);
		\draw (8.center) to (6.center);
		\draw [bend right=15] (6.center) to (7.center);
		\draw (7.center) to (9.center);
		\draw [bend left=45] (6.center) to (7.center);
		\draw [bend left=90, looseness=2.00] (10.center) to (11.center);
		\draw [bend right=90, looseness=2.00] (10.center) to (11.center);
		\draw [style=open, bend left=15, looseness=0.75] (14.center) to (15.center);
		\draw [style=open, bend right=15] (12.center) to (13.center);
		\draw [style=open, bend left, looseness=0.75] (16.center) to (17.center);
		\draw [style=dotarrow] (23.center) to (24.center);
		\draw [style=dotarrow] (26.center) to (25.center);
		\draw [style=dotarrow] (27.center) to (28.center);
		\draw [style=dotarrow] (32.center) to (31.center);
	\end{pgfonlayer}
\end{tikzpicture}
	\caption{An open-closed  surface with D-brane labels. As a morphism in $\OC$, we read the surface from left to right, i.e.\ with the source object (constituted by the incoming boundary components) on the left and the target object (constituted by the outgoing boundary components) on the right.
		We will, however, deviate from the left-to-right drawing convention at times if it simplifies the surface; for this reason, we also indicate by `in' and `out' whether a boundary component is incoming or outgoing. The source object is given by $\{c,o_1,o_2\}$ 
		(where the identification with boundary components is through the dotted arrows in the picture) plus the assignment $s(o_1)=s(o_2)=P$ and $t(o_1)=t(o_2)=Q$. The target object is $\{o'\}$ plus the assignments $s(o')=t(o')=  R$. }
	\label{figoc}
\end{figure}
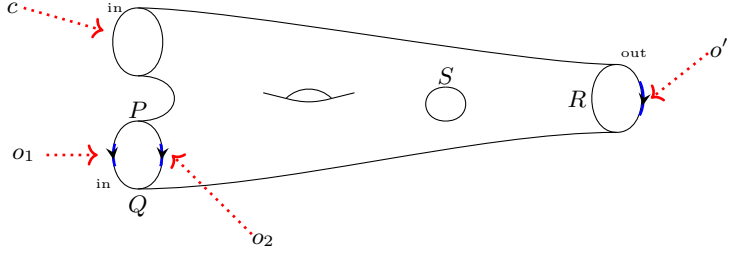

\begin{definition}[Costello \cite{costellotcft} following Getzler~\cite{getzler} and Segal~\cite{segal}]For a fixed set $\Lambda$ 
	of $D$-branes, 
	an \emph{open-closed topological conformal field theory} is a symmetric monoidal functor $\Phi : \OC \to \Ch$.
	An \emph{open topological conformal field theory}
	is a symmetric monoidal functor $\catf{O}\to\Ch$ defined only on the subcategory $\catf{O}\subset \OC$ of open bordisms.
\end{definition}

Open-closed topological conformal field theories
are a differential graded generalization
of ordinary vector space-valued 
two-dimensional (open-closed) topological field theories.
The latter can be constructed and 
 classified in terms of symmetric and 
commutative Frobenius algebras, 
see \cite{kock,laudapfeiffer} for the precise statements.

In \cite{costellotcft},
Costello proves that one may construct an \emph{open} topological conformal field theory from a (linear) Calabi-Yau category (Costello actually considers differential graded Calabi-Yau categories, but we just need the linear case).
By homotopy left Kan extension, one obtains an open-closed topological conformal field theory:

\spaceplease
\begin{theorem}[Costello \cite{costellotcft}]\label{thmcostello}\begin{pnum}\item
	Any linear Calabi-Yau category $\cat{A}$ gives rise to
	an open-closed topological conformal field theory $\OC \to \Ch$ with the object set of $\cat{A}$ as the set of $D$-branes.
	\item The value of this field theory on the circle is equivalent to the Hochschild complex of $\cat{A}$. \end{pnum}
\end{theorem}

In particular, if $\cat{M}_{p,q}$ is the moduli space of Riemann surfaces with $p$ incoming closed $(p\ge 1)$, $q$ outgoing closed and no open and no free boundary components, there are maps 
\begin{align} C_*(\cat{M}_{p,q};k) \otimes \left(    \lint^{a \in \cat{A}}\cat{A}(a,a)\right)^{\otimes p} \to \left(    \lint^{a \in \cat{A}}\cat{A}(a,a)\right)^{\otimes q} \ .  
\end{align}
We refer to \cite{dva} for details on the homotopy coends appearing here.
In the present text, this is only needed to a very limited extent.
It suffices to know that in degree zero the complex $\lint^{a \in \cat{A}} \cat{A}(a,a)$ is given by
$\bigoplus_{a\in\cat{A}} \cat{A}(a,a)$.

\begin{remark}
	In \cite{costellotcft}, it is actually required that $k$ has characteristic zero, but an extension to fields of arbitrary characteristic is given in \cite{egas,wahlwesterland}.
	\end{remark}

From Corollary~\ref{cortheiso} and Costello's result, we immediately obtain:

\begin{corollary}\label{corollarytracefieldtheory}
	For a finite category $\cat{C}$, any trivialization of the 
	Nakayama functor
$
	\naka :\cat{C}\to\cat{C}$
	yields a Calabi-Yau structure on $\Proj \cat{C}$
	and hence gives rise to a topological conformal field theory $\Phi_\cat{C} : \OC \to \Ch$ with set of D-branes given by the set of projective objects of $\cat{C}$. 
\end{corollary}

\begin{remark}\label{remtraces}
	The evaluation of the field theory $\Phi_\cat{C}$ on the disk 
	with one incoming open boundary interval whose complementing free boundary carries the D-brane label $P\in\Proj\cat{C}$
	\begin{equation} \Phi_\cat{C} \left( \begin{tikzpicture}[scale=0.7]
	\begin{pgfonlayer}{nodelayer}
		\node [style=none] (21) at (0, 2) {};
		\node [style=none] (22) at (0, -1) {};
		\node [style=none] (25) at (-1.5, 1) {};
		\node [style=none] (26) at (-1.5, 0) {};
		\node [style=none] (27) at (2, 0.5) {{\small $P$}};
		\node [style=none] (28) at (1, 2) {};
		\node [style=none] (29) at (-2.25, 0.5) {\IN};
	\end{pgfonlayer}
	\begin{pgfonlayer}{edgelayer}
		\draw [bend left=90, looseness=1.75] (21.center) to (22.center);
		\draw [bend right=90, looseness=1.75] (21.center) to (22.center);
		\draw [style=open, bend right=15, looseness=0.75] (25.center) to (26.center);
	\end{pgfonlayer}
\end{tikzpicture}  \right)\  : \  \cat{C}(P,P)\to k 
	\end{equation}
	is exactly the trace function of the Calabi-Yau structure from Definition~\ref{deftracefc} (this follows directly from Costello's construction), while
	for $P,Q\in \Proj\cat{C}$,
	the map 
		\begin{equation} \Phi_\cat{C} \left( \begin{tikzpicture}[scale=0.7]
	\begin{pgfonlayer}{nodelayer}
		\node [style=none] (21) at (0, 2) {};
		\node [style=none] (22) at (0, -1) {};
		\node [style=none] (25) at (-1, 1.75) {};
		\node [style=none] (26) at (-1.5, 1) {};
		\node [style=none] (27) at (0, 2.5) {{\small $P$}};
		\node [style=none] (28) at (1, 2) {};
		\node [style=none] (29) at (-1.75, 1.75) {\IN};
		\node [style=none] (30) at (-1.5, 0) {};
		\node [style=none] (31) at (-1, -0.75) {};
		\node [style=none] (32) at (1.25, 1.5) {};
		\node [style=none] (33) at (1.25, -0.5) {};
		\node [style=none] (34) at (-2, -0.75) {\IN};
		\node [style=none] (35) at (2.5, 0.5) {\OUT};
		\node [style=none] (36) at (0, -1.75) {{\small $P$}};
		\node [style=none] (37) at (-2.25, 0.5) {{\small $Q$}};
	\end{pgfonlayer}
	\begin{pgfonlayer}{edgelayer}
		\draw [bend left=90, looseness=1.75] (21.center) to (22.center);
		\draw [bend right=90, looseness=1.75] (21.center) to (22.center);
		\draw [style=open, bend right=15, looseness=0.75] (25.center) to (26.center);
		\draw [style=open, bend right=15, looseness=1.25] (30.center) to (31.center);
		\draw [style=open, bend left] (32.center) to (33.center);
	\end{pgfonlayer}
\end{tikzpicture}  \right)\  : \  \cat{C}(P,Q) \otimes \cat{C}(Q,P) \to \cat{C}(P,P) 
	\end{equation}
is the composition over $Q$. 
	\end{remark}

The construction of Corollary~\ref{corollarytracefieldtheory} does not do much:
It just translates a trivialization of $\naka$ to a Calabi-Yau structure and then a topological conformal field theory, and in fact, this construction 
will only be of limited use to us since we want to treat finite tensor categories and not just linear categories.
Fortunately, we can give a natural refinement: The construction from Corollary~\ref{corollarytracefieldtheory} becomes more meaningful
in the context of finite tensor categories if $\naka$ is trivialized not just as a linear functor, \emph{but as a right $\cat{C}$-module functor relative to a pivotal structure}. Let us define what we mean by that:

\begin{definition}\label{defsymfrob}
	For any finite tensor category $\cat{C}$ and a pivotal structure $\omega: -^{\vee \vee} \cong \id_\cat{C}$,
	denote by $(\id_\cat{C},\omega)$ the identity functor endowed with the structure of a right $\cat{C}$-module functor $\cat{C}_\cat{C}\to\cat{C}^{\vee\vee}$ by means of $\omega$.
	We refer to an isomorphism $\naka \cong (\id_\cat{C},\omega)$
	of right $\cat{C}$-module functors as a trivialization of $\naka$ as a right $\cat{C}$-module functor relative to $\omega$.
	We define a \emph{symmetric Frobenius structure} on a finite tensor category $\cat{C}$
	as a trivialization of $\naka$ as right $\cat{C}$-module functor relative to a pivotal structure, where the pivotal structure is part of the data.
	\end{definition}

\begin{remark}\label{remunpacksymFrob}
	Thanks to Theorem~\ref{thmnakamodule},
	 a symmetric Frobenius structure on a finite tensor category is a pivotal structure plus a trivialization of $D$. We use the term symmetric Frobenius structure not only as convenient shorthand for the rather clumsy description of `pivotal unimodular finite tensor category with a trivialization of the distinguished invertible object as part of the data', but also for a deeper reason: The symmetric Frobenius algebra structure on a finite tensor category allows us to write $\cat{C}$, as a linear category, as modules over a symmetric Frobenius algebra. This can be 
	read off from the  Corollary~\ref{cordualhom} 
	because it provides
	canonical natural isomorphisms
		\begin{align} \cat{C}(X,Y_\bullet)^* \cong \cat{C}(Y_\bullet, X  ) \ ,  \end{align}
	where $X\in\cat{C}$
	and $Y_\bullet$ is a projective resolution of $Y\in\cat{C}$. 
	However,
	a finite tensor category with symmetric Frobenius structure requires
	a compatibility of the monoidal structure with
	the trivialization of $\naka$. It is not just an identification of $\cat{C}$, as \emph{linear category}, with modules over a symmetric Frobenius algebra. In the latter sense, the notion is used in \cite{shimizucoend}. 
	\end{remark}

\begin{definition}\label{deftracefieldtheory}
	Let $\cat{C}$ be a finite tensor category with symmetric Frobenius structure.
	For the trivialization of the right Nakayama functor $\naka$ (that $\cat{C}$ 
	by Definition~\ref{defsymfrob}
	 comes equipped with), we refer to the topological conformal field theory 
	$\Phi_\cat{C} : \OC \to \Ch$ built from \emph{this particular} trivialization in the sense of Corollary~\ref{corollarytracefieldtheory}
	as the \emph{trace field theory of $\cat{C}$}. 
	\end{definition}

The name is  chosen because $\Phi_\cat{C}$ 
does not only recover the trace of the Calabi-Yau structure by Remark~\ref{remtraces}, but can also be  recovered from the trace itself.

\begin{theorem}\label{thmtracefieldtheory}
Let $\cat{C}$ be a finite tensor category with symmetric Frobenius structure and $\Phi_\cat{C}:\OC\to \Ch$ its trace field theory.
The evaluation of $\Phi_\cat{C}$ on the disk 
with one incoming open boundary interval whose complementing free boundary carries the label $P\in\Proj\cat{C}$
\begin{equation} \Phi_\cat{C} \left( \begin{tikzpicture}[scale=0.7]
	\begin{pgfonlayer}{nodelayer}
		\node [style=none] (21) at (0, 2) {};
		\node [style=none] (22) at (0, -1) {};
		\node [style=none] (25) at (-1.5, 1) {};
		\node [style=none] (26) at (-1.5, 0) {};
		\node [style=none] (27) at (2, 0.5) {{\small $P$}};
		\node [style=none] (28) at (1, 2) {};
		\node [style=none] (29) at (-2.25, 0.5) {\IN};
	\end{pgfonlayer}
	\begin{pgfonlayer}{edgelayer}
		\draw [bend left=90, looseness=1.75] (21.center) to (22.center);
		\draw [bend right=90, looseness=1.75] (21.center) to (22.center);
		\draw [style=open, bend right=15, looseness=0.75] (25.center) to (26.center);
	\end{pgfonlayer}
\end{tikzpicture}  \right)\  : \  \cat{C}(P,P)\to k 
\end{equation}
is a right modified trace, while the evaluation of $\Phi_\cat{C}$ on the cylinder with one incoming open boundary interval with complementing free boundary label $P\in\Proj\cat{C}$ and one outgoing closed boundary circle
\begin{align}
	\Phi_\cat{C} \left( \begin{tikzpicture}[scale=0.7]
	\begin{pgfonlayer}{nodelayer}
		\node [style=none] (0) at (2.5, 24) {};
		\node [style=none] (1) at (2.5, 21) {};
		\node [style=none] (2) at (1.5, 23) {};
		\node [style=none] (3) at (1.5, 21.75) {};
		\node [style=none] (4) at (4.25, 22.5) {{\small $P$}};
		\node [style=none] (5) at (6.5, 24) {};
		\node [style=none] (6) at (6.5, 21) {};
		\node [style=none] (7) at (0.5, 22.5) {\IN};
		\node [style=none] (8) at (8, 24) {\OUT};
	\end{pgfonlayer}
	\begin{pgfonlayer}{edgelayer}
		\draw [bend left=90, looseness=1.25] (0.center) to (1.center);
		\draw [bend right=90, looseness=1.25] (0.center) to (1.center);
		\draw [style=open, bend right, looseness=0.75] (2.center) to (3.center);
		\draw [bend left=90, looseness=1.25] (5.center) to (6.center);
		\draw [bend right=90, looseness=1.25] (5.center) to (6.center);
		\draw (0.center) to (5.center);
		\draw (1.center) to (6.center);
	\end{pgfonlayer}
\end{tikzpicture}  \right)\  : \  \cat{C}(P,P)\to \lint^{P\in\Proj\cat{C}} \cat{C}(P,P) \label{eqnht}
\end{align}
agrees, after taking zeroth homology, with the Hattori-Stallings trace of $\cat{C}$.
\end{theorem}

\begin{proof}
	The statement about the modified trace can be extracted from Theorem~\ref{thmmtrace}
	and Remark~\ref{remtraces}.

	For the second statement,
	we first observe that by the construction of $\Phi_\cat{C}$
	the map~\eqref{eqnht} is just the inclusion of $\cat{C}(P,P)$ into the direct sum $\bigoplus_{P\in\Proj\cat{C}} \cat{C}(P,P)$, which is the degree zero term of the Hochschild complex. After taking homology, we get the natural map
	$\cat{C}(P,P)\to HH_0(\cat{C})$, i.e.\ the quotient map projecting to zeroth Hochschild homology, and hence the Hattori-Stallings trace for $\cat{C}$.
	The connection to the traditional Hattori-Stallings trace \cite{hattori,stallings}
	 uses  that by writing $\cat{C}$, as a linear category,
	  as finite-dimensional modules over a finite-dimensional algebra $A$ (which we can always do), the zeroth homology
	$HH_0(\cat{C})$ is isomorphic to the zeroth Hochschild homology $HH_0(A)=A/[A,A]$
	of $A$. This is a consequence of the Agreement Principle
	of McCarthy \cite{mcarthy} and Keller \cite{keller}, see also \cite[Section~3.2]{dva} for this principle in the context of finite tensor categories.
	\end{proof}

We formulate the result in Theorem~\ref{thmtracefieldtheory}
topologically (although the  Theorem~\ref{thmmtrace} that it relies on is purely algebraic) because, instead of traces, we will  in Section~\ref{sectrivE2}  use the trace field theory as an efficient tool
for computations.

\begin{remark}
	The reader should appreciate that the trace field $\Phi_\cat{C}$ is defined through a specific trivialization of the Nakayama functor and not by choosing a modified trace (although Theorem~\ref{thmtracefieldtheory} tells us that we could have done that).
	This has the advantage that, through the closed formula for $\naka$, the trace field theory $\Phi_\cat{C}$ becomes very accessible.
	In fact, we will rely on the particular definition of $\Phi_\cat{C}$ given above in future work.
\end{remark}

\spaceplease
\section{The block diagonal product on Hochschild chains\label{sectrivE2}}
We will now see how we can profit from the topological description of traces. First we define a multiplication by evaluation on the pair of pants:

\begin{definition}\label{defstarprod}
Let $\cat{C}$ be a finite tensor category with symmetric Frobenius structure and $\Phi_\cat{C}:\OC\to \Ch$ the associated trace field theory.
Then we define the \emph{block diagonal $\star$-product} on the Hochschild complex $\lint^{X\in\Proj\cat{C}} \cat{C}(X,X)$ by
\begin{align}
\star := \Phi_\cat{C}\left(\begin{tikzpicture}[scale=0.4]
	\begin{pgfonlayer}{nodelayer}
		\node [style=none] (0) at (1.5, 12.5) {};
		\node [style=none] (1) at (1.5, 9.5) {};
		\node [style=none] (2) at (1.5, 6.5) {};
		\node [style=none] (3) at (1.5, 3.5) {};
		\node [style=none] (4) at (8.5, 9.5) {};
		\node [style=none] (5) at (8.5, 6.5) {};
		\node [style=none] (6) at (0.5, 10.25) {};
		\node [style=none] (7) at (-0.5, 12) {\IN};
		\node [style=none] (8) at (-0.5, 4) {\IN};
		\node [style=none] (9) at (10.75, 9) {\OUT};
	\end{pgfonlayer}
	\begin{pgfonlayer}{edgelayer}
		\draw [bend left=90, looseness=1.25] (0.center) to (1.center);
		\draw [bend right=90, looseness=1.25] (0.center) to (1.center);
		\draw [bend left=90, looseness=1.25] (2.center) to (3.center);
		\draw [bend right=90, looseness=1.25] (2.center) to (3.center);
		\draw [bend left=90, looseness=1.25] (4.center) to (5.center);
		\draw [bend right=90, looseness=1.25] (4.center) to (5.center);
		\draw [bend right=90, looseness=2.75] (2.center) to (1.center);
		\draw [in=180, out=0, looseness=0.50] (0.center) to (4.center);
		\draw [in=-180, out=0, looseness=0.75] (3.center) to (5.center);
	\end{pgfonlayer}
\end{tikzpicture}\right)\ : \ \lint^{X\in\Proj\cat{C}} \cat{C}(X,X) \otimes \lint^{X\in\Proj\cat{C}} \cat{C}(X,X)\to \lint^{X\in\Proj\cat{C}} \cat{C}(X,X) \, . \label{eqnstarproduct}
\end{align}
\end{definition}
The sense in which the product $\star$ is block diagonal will be discussed in Proposition~\ref{propdiag}.

The results of 
Wahl and Westerland on the product obtained from a topological conformal field theory in \cite[Section~6]{wahlwesterland} imply
that, up to homotopy, the multiplication is concentrated in degree zero (they prove it for symmetric Frobenius algebras, but their proof carries over to our situation). 	
They also give a formula for the degree zero part of the homotopy commutative multiplication. We will below give a slightly different formula which, when working with a Calabi-Yau category instead of a symmetric Frobenius algebra, is a little more convenient.

\begin{lemma}\label{lemmadeligneinH0}
	The degree zero part
	\begin{align}
	\star\ :	\bigoplus_{P,Q \in \Proj \cat{C}} \cat{C}(P,P)\otimes\cat{C}(Q,Q) \to \bigoplus_{P\in \Proj \cat{C}} \cat{C}(P,P) 
	\end{align}
	of the product from Definition~\ref{eqnstarproduct} is given on the summand $\cat{C}(P,P)\otimes\cat{C}(Q,Q)$, up to boundary,  by the linear map
	\begin{equation}
	\Phi_\cat{C}\left(\begin{tikzpicture}[scale=0.6]
	\begin{pgfonlayer}{nodelayer}
		\node [style=none] (21) at (0, 1.5) {};
		\node [style=none] (22) at (0, -1.5) {};
		\node [style=none] (25) at (-3.5, 1.25) {};
		\node [style=none] (26) at (-3.5, -1) {};
		\node [style=none] (27) at (0, 2) {{\small $Q$}};
		\node [style=none] (28) at (0, -4.25) {{\small $P$}};
		\node [style=none] (31) at (0, 4.25) {{\small $P$}};
		\node [style=none] (32) at (1.5, 0.5) {};
		\node [style=none] (33) at (1.5, -0.5) {};
		\node [style=none] (34) at (0, 3.5) {};
		\node [style=none] (35) at (0, -3.5) {};
		\node [style=none] (36) at (3.5, 1.25) {};
		\node [style=none] (37) at (3.5, -1.25) {};
		\node [style=none] (38) at (-4, -3) {};
		\node [style=none] (39) at (-4, -4) {};
		\node [style=none] (40) at (-4.5, 0) {\IN};
		\node [style=none] (41) at (5, 0) {\OUT};
		\node [style=none] (42) at (-5, -3.5) {\IN};
	\end{pgfonlayer}
	\begin{pgfonlayer}{edgelayer}
		\draw [bend right=90, looseness=1.75] (21.center) to (22.center);
		\draw [style=open, bend right=15, looseness=0.75] (25.center) to (26.center);
		\draw [bend right=90, looseness=1.75] (34.center) to (35.center);
		\draw [bend left=90, looseness=1.75] (34.center) to (35.center);
		\draw [style=open, in=75, out=-75, looseness=0.75] (36.center) to (37.center);
		\draw [style=over] (38.center) to (32.center);
		\draw [style=over] (39.center) to (33.center);
		\draw [in=120, out=0] (21.center) to (32.center);
		\draw [in=-90, out=0, looseness=0.75] (22.center) to (33.center);
		\draw (38.center) to (39.center);
		\draw [style=open] (38.center) to (39.center);
	\end{pgfonlayer}
\end{tikzpicture}\right) \  : \ \cat{C}(P,P)\otimes\cat{C}(Q,Q) \to \cat{C}(P,P) \ . \label{multviaPhieqn}
	\end{equation}
\end{lemma}
	
\begin{proof} Let $P$ and $Q$ be projective objects in $\cat{C}$.
	The following morphisms in $\OC$ can be deformed into each other and hence represent homologous zero chains:
	\begin{align}{\footnotesize 
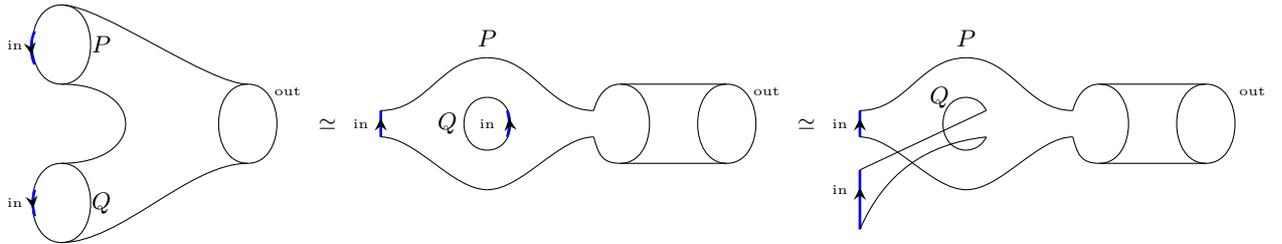
\begin{tikzpicture}[scale=0.7]
	\begin{pgfonlayer}{nodelayer}
		\node [style=none] (0) at (-11, 6) {};
		\node [style=none] (1) at (-11, 3) {};
		\node [style=none] (2) at (-11, 0) {};
		\node [style=none] (3) at (-11, -3) {};
		\node [style=none] (4) at (-4, 3) {};
		\node [style=none] (5) at (-4, 0) {};
		\node [style=none] (12) at (-12, -1) {};
		\node [style=none] (13) at (-12, -2) {};
		\node [style=none] (14) at (-12, 5) {};
		\node [style=none] (15) at (-12, 3.75) {};
		\node [style=none] (18) at (-9.5, 4.5) {{\small $P$}};
		\node [style=none] (19) at (-9.5, -1.5) {{\small $Q$}};
		\node [style=none] (20) at (14, 3) {};
		\node [style=none] (21) at (14, 0) {};
		\node [style=none] (22) at (10, 3) {};
		\node [style=none] (23) at (10, 0) {};
		\node [style=none] (24) at (9, 2) {};
		\node [style=none] (25) at (9, 1) {};
		\node [style=none] (26) at (5, 2.5) {};
		\node [style=none] (27) at (5, 0.5) {};
		\node [style=none] (28) at (5.75, 2) {};
		\node [style=none] (29) at (5.75, 1) {};
		\node [style=none] (30) at (1, 2) {};
		\node [style=none] (31) at (1, 1) {};
		\node [style=none] (32) at (5, 4) {};
		\node [style=none] (33) at (5, -1) {};
		\node [style=none] (34) at (-1, 1.5) {$\simeq$};
		\node [style=none] (35) at (5, 4.75) {{\small $P$}};
		\node [style=none] (36) at (3.5, 1.5) {{\small $Q$}};
		\node [style=none] (68) at (32, 3) {};
		\node [style=none] (69) at (32, 0) {};
		\node [style=none] (70) at (28, 3) {};
		\node [style=none] (71) at (28, 0) {};
		\node [style=none] (72) at (27, 2) {};
		\node [style=none] (73) at (27, 1) {};
		\node [style=none] (74) at (23, 2.5) {};
		\node [style=none] (75) at (23, 0.5) {};
		\node [style=none] (76) at (19, -0.25) {};
		\node [style=none] (77) at (19, -2.5) {};
		\node [style=none] (78) at (19, 2) {};
		\node [style=none] (79) at (19, 1) {};
		\node [style=none] (80) at (23, 4) {};
		\node [style=none] (81) at (23, -1) {};
		\node [style=none] (82) at (23, 4.75) {{\small $P$}};
		\node [style=none] (83) at (22, 2.5) {{\small $Q$}};
		\node [style=none] (84) at (23.75, 2) {};
		\node [style=none] (85) at (23.75, 1) {};
		\node [style=none] (86) at (17, 1.5) {$\simeq$};
		\node [style=none] (87) at (-12.75, 4.5) {\IN};
		\node [style=none] (88) at (-12.75, -1.5) {\IN};
		\node [style=none] (89) at (-2.5, 2.75) {\OUT};
		\node [style=none] (90) at (0.25, 1.5) {\IN};
		\node [style=none] (91) at (5, 1.5) {\IN};
		\node [style=none] (92) at (15.5, 2.75) {\OUT};
		\node [style=none] (93) at (18.25, 1.5) {\IN};
		\node [style=none] (94) at (18.25, -1) {\IN};
		\node [style=none] (95) at (33.75, 2.75) {\OUT};
		\node [style=none] (96) at (23.75, 2) {};
		\node [style=none] (97) at (23.75, 1) {};
	\end{pgfonlayer}
	\begin{pgfonlayer}{edgelayer}
		\draw [bend left=90, looseness=1.25] (0.center) to (1.center);
		\draw [bend right=90, looseness=1.25] (0.center) to (1.center);
		\draw [bend left=90, looseness=1.25] (2.center) to (3.center);
		\draw [bend right=90, looseness=1.25] (2.center) to (3.center);
		\draw [bend left=90, looseness=1.25] (4.center) to (5.center);
		\draw [bend right=90, looseness=1.25] (4.center) to (5.center);
		\draw [bend right=90, looseness=2.75] (2.center) to (1.center);
		\draw [in=180, out=0, looseness=0.50] (0.center) to (4.center);
		\draw [in=-180, out=0, looseness=0.75] (3.center) to (5.center);
		\draw [style=open, bend right, looseness=0.75] (14.center) to (15.center);
		\draw [style=open, bend right=15] (12.center) to (13.center);
		\draw [bend left=90, looseness=1.25] (20.center) to (21.center);
		\draw [bend right=90, looseness=1.25] (20.center) to (21.center);
		\draw [bend left=90, looseness=1.25] (22.center) to (23.center);
		\draw (22.center) to (20.center);
		\draw (23.center) to (21.center);
		\draw [in=-75, out=-180, looseness=1.25] (23.center) to (25.center);
		\draw [in=75, out=180] (22.center) to (24.center);
		\draw [bend left=90, looseness=1.50] (26.center) to (27.center);
		\draw [bend right=90, looseness=1.50] (26.center) to (27.center);
		\draw [style=open, bend right=15] (29.center) to (28.center);
		\draw [style=open] (31.center) to (30.center);
		\draw [in=-180, out=0] (30.center) to (32.center);
		\draw [in=180, out=0, looseness=0.75] (31.center) to (33.center);
		\draw [in=180, out=0] (32.center) to (24.center);
		\draw [in=-180, out=0, looseness=0.75] (33.center) to (25.center);
		\draw [bend left=90, looseness=1.25] (68.center) to (69.center);
		\draw [bend right=90, looseness=1.25] (68.center) to (69.center);
		\draw [bend left=90, looseness=1.25] (70.center) to (71.center);
		\draw (70.center) to (68.center);
		\draw (71.center) to (69.center);
		\draw [in=-75, out=-180, looseness=1.25] (71.center) to (73.center);
		\draw [in=75, out=180] (70.center) to (72.center);
		\draw [bend right=90, looseness=1.50] (74.center) to (75.center);
		\draw [style=open] (77.center) to (76.center);
		\draw [style=open] (79.center) to (78.center);
		\draw [in=-180, out=0] (78.center) to (80.center);
		\draw [in=180, out=0, looseness=0.75] (79.center) to (81.center);
		\draw [in=180, out=0] (80.center) to (72.center);
		\draw [in=-180, out=0, looseness=0.75] (81.center) to (73.center);
		\draw [bend left] (74.center) to (84.center);
		\draw [bend right] (75.center) to (85.center);
		\draw (76.center) to (96.center);
		\draw [bend left] (77.center) to (97.center);
	\end{pgfonlayer}
\end{tikzpicture}}
	\end{align}
	But this means that the square in $\OC$
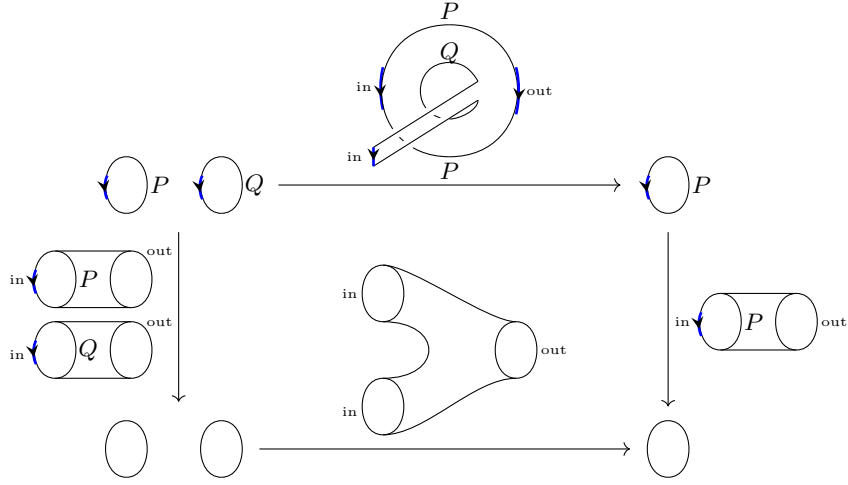
	\begin{align} \begin{tikzpicture}[scale=.5]
	\begin{pgfonlayer}{nodelayer}
		\node [style=none] (0) at (-2.5, 23.25) {};
		\node [style=none] (1) at (-2.5, 20.25) {};
		\node [style=none] (2) at (-2.5, 17.25) {};
		\node [style=none] (3) at (-2.5, 14.25) {};
		\node [style=none] (4) at (4.5, 20.25) {};
		\node [style=none] (5) at (4.5, 17.25) {};
		\node [style=none] (15) at (-3.5, 21) {};
		\node [style=none] (72) at (-8, 27.5) {};
		\node [style=none] (73) at (10, 27.5) {};
		\node [style=none] (74) at (-13.25, 25) {};
		\node [style=none] (75) at (-13.25, 16) {};
		\node [style=none] (76) at (12.5, 25) {};
		\node [style=none] (77) at (12.5, 15.75) {};
		\node [style=none] (78) at (-9, 13.5) {};
		\node [style=none] (79) at (10.5, 13.5) {};
		\node [style=none] (80) at (-16, 29) {};
		\node [style=none] (81) at (-16, 26) {};
		\node [style=none] (82) at (-17, 28) {};
		\node [style=none] (83) at (-17, 26.75) {};
		\node [style=none] (84) at (-14.25, 27.5) {{\small $P$}};
		\node [style=none] (85) at (-11, 29) {};
		\node [style=none] (86) at (-11, 26) {};
		\node [style=none] (87) at (-12, 28) {};
		\node [style=none] (88) at (-12, 26.75) {};
		\node [style=none] (89) at (-9.25, 27.5) {{\small $Q$}};
		\node [style=none] (90) at (12.5, 29) {};
		\node [style=none] (91) at (12.5, 26) {};
		\node [style=none] (92) at (11.5, 28) {};
		\node [style=none] (93) at (11.5, 26.75) {};
		\node [style=none] (94) at (14.25, 27.5) {{\small $P$}};
		\node [style=none] (95) at (-16, 15) {};
		\node [style=none] (96) at (-16, 12) {};
		\node [style=none] (100) at (1, 34) {};
		\node [style=none] (101) at (1, 31) {};
		\node [style=none] (102) at (-2.5, 33.75) {};
		\node [style=none] (103) at (-2.5, 31.5) {};
		\node [style=none] (104) at (1, 34.5) {{\small $Q$}};
		\node [style=none] (105) at (1, 28.25) {{\small $P$}};
		\node [style=none] (106) at (1, 36.75) {{\small $P$}};
		\node [style=none] (107) at (2.5, 33) {};
		\node [style=none] (108) at (2.5, 32) {};
		\node [style=none] (109) at (1, 36) {};
		\node [style=none] (110) at (1, 29) {};
		\node [style=none] (111) at (4.5, 33.75) {};
		\node [style=none] (112) at (4.5, 31.25) {};
		\node [style=none] (113) at (-3, 29.5) {};
		\node [style=none] (114) at (-3, 28.5) {};
		\node [style=none] (115) at (-19.75, 20.25) {};
		\node [style=none] (116) at (-19.75, 17.25) {};
		\node [style=none] (117) at (-20.75, 19.25) {};
		\node [style=none] (118) at (-20.75, 18) {};
		\node [style=none] (119) at (-18, 18.75) {{\small $Q$}};
		\node [style=none] (120) at (-15.75, 20.25) {};
		\node [style=none] (121) at (-15.75, 17.25) {};
		\node [style=none] (122) at (-19.75, 24) {};
		\node [style=none] (123) at (-19.75, 21) {};
		\node [style=none] (124) at (-20.75, 23) {};
		\node [style=none] (125) at (-20.75, 21.75) {};
		\node [style=none] (126) at (-18, 22.5) {{\small $P$}};
		\node [style=none] (127) at (-15.75, 24) {};
		\node [style=none] (128) at (-15.75, 21) {};
		\node [style=none] (129) at (-11, 15) {};
		\node [style=none] (130) at (-11, 12) {};
		\node [style=none] (131) at (12.5, 15) {};
		\node [style=none] (132) at (12.5, 12) {};
		\node [style=none] (133) at (15.25, 21.75) {};
		\node [style=none] (134) at (15.25, 18.75) {};
		\node [style=none] (135) at (14.25, 20.75) {};
		\node [style=none] (136) at (14.25, 19.5) {};
		\node [style=none] (137) at (17, 20.25) {{\small $P$}};
		\node [style=none] (138) at (19.25, 21.75) {};
		\node [style=none] (139) at (19.25, 18.75) {};
		\node [style=none] (140) at (-21.75, 22.5) {\IN};
		\node [style=none] (141) at (-21.75, 18.5) {\IN};
		\node [style=none] (142) at (-14.25, 24) {\OUT};
		\node [style=none] (143) at (-14.25, 20.25) {\OUT};
		\node [style=none] (144) at (-4.25, 21.75) {\IN};
		\node [style=none] (145) at (-4.25, 15.5) {\IN};
		\node [style=none] (146) at (6.5, 18.75) {\OUT};
		\node [style=none] (147) at (13.25, 20.25) {\IN};
		\node [style=none] (148) at (21.25, 20.25) {\OUT};
		\node [style=none] (149) at (5.75, 32.5) {\OUT};
		\node [style=none] (150) at (-4, 29) {\IN};
		\node [style=none] (151) at (-3.25, 32.75) {};
		\node [style=none] (152) at (-3.5, 32.75) {\IN};
	\end{pgfonlayer}
	\begin{pgfonlayer}{edgelayer}
		\draw [bend left=90, looseness=1.25] (0.center) to (1.center);
		\draw [bend right=90, looseness=1.25] (0.center) to (1.center);
		\draw [bend left=90, looseness=1.25] (2.center) to (3.center);
		\draw [bend right=90, looseness=1.25] (2.center) to (3.center);
		\draw [bend left=90, looseness=1.25] (4.center) to (5.center);
		\draw [bend right=90, looseness=1.25] (4.center) to (5.center);
		\draw [bend right=90, looseness=2.75] (2.center) to (1.center);
		\draw [in=180, out=0, looseness=0.50] (0.center) to (4.center);
		\draw [in=-180, out=0, looseness=0.75] (3.center) to (5.center);
		\draw [style=end arrow] (72.center) to (73.center);
		\draw [style=end arrow] (76.center) to (77.center);
		\draw [style=end arrow] (78.center) to (79.center);
		\draw [style=end arrow] (74.center) to (75.center);
		\draw [bend left=90, looseness=1.25] (80.center) to (81.center);
		\draw [bend right=90, looseness=1.25] (80.center) to (81.center);
		\draw [style=open, bend right, looseness=0.75] (82.center) to (83.center);
		\draw [bend left=90, looseness=1.25] (85.center) to (86.center);
		\draw [bend right=90, looseness=1.25] (85.center) to (86.center);
		\draw [style=open, bend right, looseness=0.75] (87.center) to (88.center);
		\draw [bend left=90, looseness=1.25] (90.center) to (91.center);
		\draw [bend right=90, looseness=1.25] (90.center) to (91.center);
		\draw [style=open, bend right, looseness=0.75] (92.center) to (93.center);
		\draw [bend left=90, looseness=1.25] (95.center) to (96.center);
		\draw [bend right=90, looseness=1.25] (95.center) to (96.center);
		\draw [bend right=90, looseness=1.75] (100.center) to (101.center);
		\draw [style=open, bend right=15, looseness=0.75] (102.center) to (103.center);
		\draw [bend right=90, looseness=1.75] (109.center) to (110.center);
		\draw [bend left=90, looseness=1.75] (109.center) to (110.center);
		\draw [style=open, in=75, out=-75, looseness=0.75] (111.center) to (112.center);
		\draw [style=over] (113.center) to (107.center);
		\draw [style=over] (114.center) to (108.center);
		\draw [in=120, out=0] (100.center) to (107.center);
		\draw [in=-90, out=0, looseness=0.75] (101.center) to (108.center);
		\draw (113.center) to (114.center);
		\draw [style=open] (113.center) to (114.center);
		\draw [bend left=90, looseness=1.25] (115.center) to (116.center);
		\draw [bend right=90, looseness=1.25] (115.center) to (116.center);
		\draw [style=open, bend right, looseness=0.75] (117.center) to (118.center);
		\draw [bend left=90, looseness=1.25] (120.center) to (121.center);
		\draw [bend right=90, looseness=1.25] (120.center) to (121.center);
		\draw (115.center) to (120.center);
		\draw (116.center) to (121.center);
		\draw [bend left=90, looseness=1.25] (122.center) to (123.center);
		\draw [bend right=90, looseness=1.25] (122.center) to (123.center);
		\draw [style=open, bend right, looseness=0.75] (124.center) to (125.center);
		\draw [bend left=90, looseness=1.25] (127.center) to (128.center);
		\draw [bend right=90, looseness=1.25] (127.center) to (128.center);
		\draw (122.center) to (127.center);
		\draw (123.center) to (128.center);
		\draw [bend left=90, looseness=1.25] (129.center) to (130.center);
		\draw [bend right=90, looseness=1.25] (129.center) to (130.center);
		\draw [bend left=90, looseness=1.25] (131.center) to (132.center);
		\draw [bend right=90, looseness=1.25] (131.center) to (132.center);
		\draw [bend left=90, looseness=1.25] (133.center) to (134.center);
		\draw [bend right=90, looseness=1.25] (133.center) to (134.center);
		\draw [style=open, bend right, looseness=0.75] (135.center) to (136.center);
		\draw [bend left=90, looseness=1.25] (138.center) to (139.center);
		\draw [bend right=90, looseness=1.25] (138.center) to (139.center);
		\draw (133.center) to (138.center);
		\draw (134.center) to (139.center);
	\end{pgfonlayer}
\end{tikzpicture}
	\end{align}
	commutes up to boundary.  If we apply $\Phi_\cat{C}:\OC\to\Ch$ to the square, we see that the square
		\begin{equation}
	\begin{tikzcd}
	\ar[rrr,"\eqref{multviaPhieqn}"] \ar[dd,swap] \cat{C}(P,P)\otimes\cat{C}(Q,Q) &&& \cat{C}(P,P)  \ar[dd]  \\ \\
	\lint^{X\in\Proj\cat{C}} \cat{C}(X,X) \otimes \lint^{X\in\Proj\cat{C}} \cat{C}(X,X) 	 \ar[rrr,swap,"\star"]    &&& \lint^{X\in\Proj\cat{C}} \cat{C}(X,X)  \\
	\end{tikzcd}
	\end{equation}
	commutes up to chain homotopy, where the vertical maps are just the usual embeddings of endomorphism spaces as summands in the Hochschild complex. Since we now recover the linear map \eqref{multviaPhieqn} as the upper horizontal arrow, the assertion follows. 
	\end{proof}

In the case of one object,
using the Sweedler notation for the coproduct of a symmetric Frobenius algebra, we recover the formula of Wahl and Westerland \cite[Section~6, page~41]{wahlwesterland} up to boundary.

In the sequel, it will always be implicit that the $\star$-product is applied in degree zero (because of the fact that it only contains information in that particular degree).

\begin{proposition}%[Block diagonalization of the fusion rules]
	\label{propdiag}
	For any finite tensor category $\cat{C}$ 
	with symmetric Frobenius structure, the product $\star$ is block diagonal in the sense that it vanishes on two elements in components indexed by projective objects $P$ and $Q$ with vanishing morphism space $\cat{C}(P,Q)$;
	\begin{align} f\star g= 0 \quad 
		\text{for}\quad f \in \cat{C}(P,P)\ , \quad g \in \cat{C}(Q,Q) \quad \text{if}\quad \cat{C}(P,Q)=0 \quad \text{(or equivalently $\cat{C}(Q,P)=0$)} \ . 
	\end{align}
\end{proposition}

\begin{proof}
	From the formula for the $\star$-product given in 
	Lemma~\ref{lemmadeligneinH0} one can see  that the map describing $\star$ on $\cat{C}(P,P)\otimes\cat{C}(Q,Q)$
	 factors through $\cat{C}(P,Q)$ or $\cat{C}(Q,P)$.
\end{proof}

We now prove a formula for the $\star$-product of identity endomorphisms of two projective objects.
As a preparation, we make the following Definition:

\begin{definition}\label{handleelements}
	Let $\cat{C}$ be a finite tensor category with symmetric Frobenius structure.
	For $P,Q \in \cat{C}$,
	we define the \emph{handle element of $P$ and $Q$} as the endomorphism
	 $\xi_{P,Q} \in \cat{C}(P,P)$ obtained by evaluation of the trace field theory on the annulus:
	\begin{equation}
	\xi_{P,Q} :=    \Phi_\cat{C}\left(    \begin{tikzpicture}
	\begin{pgfonlayer}{nodelayer}
		\node [style=none] (21) at (0, 2) {};
		\node [style=none] (22) at (0, -1) {};
		\node [style=none] (28) at (-2, 0.75) {{\small $P$}};
		\node [style=none] (29) at (0, 1) {};
		\node [style=none] (30) at (0, 0) {};
		\node [style=none] (31) at (0, -0.5) {{\small $Q$}};
		\node [style=none] (32) at (1.5, 1) {};
		\node [style=none] (33) at (1.5, 0) {};
		\node [style=none] (34) at (2.25, 0.5) {\OUT};
	\end{pgfonlayer}
	\begin{pgfonlayer}{edgelayer}
		\draw [bend left=90, looseness=1.75] (21.center) to (22.center);
		\draw [bend right=90, looseness=1.75] (21.center) to (22.center);
		\draw [bend left=90, looseness=1.75] (29.center) to (30.center);
		\draw [bend right=90, looseness=1.75] (29.center) to (30.center);
		\draw [style=open, bend left=15, looseness=0.75] (32.center) to (33.center);
	\end{pgfonlayer}
\end{tikzpicture}    \right)\in\cat{C}(P,P) \ . 
	\end{equation}
	\end{definition}

\begin{remark} The name of the element $\xi_{P,Q}$ 
	is chosen
	for the following reason: 
		If we were in the situation $P=Q$, the element $\xi_{P,P}$ would be the composition `$\text{multiplication}\circ \text{comultiplication} \circ \text{unit}$' in the symmetric Frobenius algebra $\cat{C}(P,P)$. 
	In \cite[page~128]{kock}, this element is called the \emph{handle element} of the symmetric Frobenius algebra.  
	\end{remark}

\begin{theorem}\label{thmhandleelement}
Let $\cat{C}$ be a finite tensor category with symmetric Frobenius structure.

\begin{pnum}
	\item For $P,Q\in\Proj\cat{C}$,
	the $\star$-product of $\id_P$ and $\id_Q$ is the  handle element $\xi_{P,Q}$ of $P$ and $Q$, up to boundary in the Hochschild complex of $\cat{C}$;
	\begin{align} \id_P \star \id_Q \simeq \xi_{P,Q} \ . 
	\end{align}\label{tracefieldi}
	\item All handle elements in the sense of Definition~\ref{handleelements} are central in the endomorphism algebras of $\cat{C}$.\label{tracefieldii}
	\item The modified trace of the handle element is given by
	\begin{align} \trace_P \xi_{P,Q}=\dim \,\cat{C}(P,Q) \ . \label{eqntraceformula}
	\end{align}\label{tracefieldiii}
	\end{pnum}	
\end{theorem}
Of course, the numbers $\dim\,\cat{C}(P,Q)$ on the right hand side 
of \eqref{eqntraceformula} are 
the entries of
the Cartan matrix of $\cat{C}$, considered here as  elements in $k$.

	If $P$ is simple, the handle element is the number
\begin{align}
\xi_{P,Q} = \frac{\dim\, \cat{C}(P,Q)}{d^\text{m} (P)} \in k \ , 
\end{align}
where $d^\text{m} (P):=\trace_P(\id_P)\in k^\times$ is the modified dimension of $P$
(note that $\trace(\id_P)\neq 0$ is a consequence of the non-degeneracy of the trace).

\begin{proof}[{\slshape Proof of Theorem~\ref{thmhandleelement}}]
		In order to compute $\id_P \star \id_Q$ for $P,Q\in\Proj\cat{C}$, we use
	 Lemma~\ref{lemmadeligneinH0} and the functoriality of $\Phi_\cat{C}$: \begin{equation}
		\id_P\star \id_Q \simeq 	\Phi_\cat{C}\left(\begin{tikzpicture}[scale=0.6]
	\begin{pgfonlayer}{nodelayer}
		\node [style=none] (21) at (0, 1.5) {};
		\node [style=none] (22) at (0, -1.5) {};
		\node [style=none] (25) at (-3.5, 1.25) {};
		\node [style=none] (26) at (-3.5, -1) {};
		\node [style=none] (27) at (0, 2) {{\small $Q$}};
		\node [style=none] (28) at (0, -4.25) {{\small $P$}};
		\node [style=none] (31) at (0, 4.25) {{\small $P$}};
		\node [style=none] (32) at (1.5, 0.5) {};
		\node [style=none] (33) at (1.5, -0.5) {};
		\node [style=none] (34) at (0, 3.5) {};
		\node [style=none] (35) at (0, -3.5) {};
		\node [style=none] (36) at (3.5, 1.25) {};
		\node [style=none] (37) at (3.5, -1.25) {};
		\node [style=none] (38) at (-4, -3) {};
		\node [style=none] (39) at (-4, -4) {};
		\node [style=none] (40) at (-4.5, 0) {\IN};
		\node [style=none] (41) at (5, 0) {\OUT};
		\node [style=none] (42) at (-5, -3.5) {\IN};
	\end{pgfonlayer}
	\begin{pgfonlayer}{edgelayer}
		\draw [bend right=90, looseness=1.75] (21.center) to (22.center);
		\draw [style=open, bend right=15, looseness=0.75] (25.center) to (26.center);
		\draw [bend right=90, looseness=1.75] (34.center) to (35.center);
		\draw [bend left=90, looseness=1.75] (34.center) to (35.center);
		\draw [style=open, in=75, out=-75, looseness=0.75] (36.center) to (37.center);
		\draw [style=over] (38.center) to (32.center);
		\draw [style=over] (39.center) to (33.center);
		\draw [in=120, out=0] (21.center) to (32.center);
		\draw [in=-90, out=0, looseness=0.75] (22.center) to (33.center);
		\draw (38.center) to (39.center);
		\draw [style=open] (38.center) to (39.center);
	\end{pgfonlayer}
\end{tikzpicture}\right) \circ \Phi_\cat{C}\left(\begin{tikzpicture}[scale=0.8]
	\begin{pgfonlayer}{nodelayer}
		\node [style=none] (21) at (0, 2) {};
		\node [style=none] (22) at (0, 0) {};
		\node [style=none] (28) at (-1.75, 1) {{\small $P$}};
		\node [style=none] (32) at (0.75, 1.75) {};
		\node [style=none] (33) at (0.75, 0.25) {};
		\node [style=none] (34) at (0, -1) {};
		\node [style=none] (35) at (0, -3) {};
		\node [style=none] (36) at (-1.75, -2) {{\small $Q$}};
		\node [style=none] (37) at (0.75, -1.25) {};
		\node [style=none] (38) at (0.75, -2.75) {};
		\node [style=none] (39) at (1.5, 1.25) {};
		\node [style=none] (40) at (2, 1) {\OUT};
		\node [style=none] (41) at (2, -2) {\OUT};
	\end{pgfonlayer}
	\begin{pgfonlayer}{edgelayer}
		\draw [bend left=90, looseness=1.75] (21.center) to (22.center);
		\draw [bend right=90, looseness=1.75] (21.center) to (22.center);
		\draw [style=open, bend left=45] (32.center) to (33.center);
		\draw [bend left=90, looseness=1.75] (34.center) to (35.center);
		\draw [bend right=90, looseness=1.75] (34.center) to (35.center);
		\draw [style=open, bend left=45] (37.center) to (38.center);
	\end{pgfonlayer}
\end{tikzpicture}\right)= \Phi_\cat{C}\left(    \begin{tikzpicture}
	\begin{pgfonlayer}{nodelayer}
		\node [style=none] (21) at (0, 2) {};
		\node [style=none] (22) at (0, -1) {};
		\node [style=none] (28) at (-2, 0.75) {{\small $P$}};
		\node [style=none] (29) at (0, 1) {};
		\node [style=none] (30) at (0, 0) {};
		\node [style=none] (31) at (0, -0.5) {{\small $Q$}};
		\node [style=none] (32) at (1.5, 1) {};
		\node [style=none] (33) at (1.5, 0) {};
		\node [style=none] (34) at (2.25, 0.5) {\OUT};
	\end{pgfonlayer}
	\begin{pgfonlayer}{edgelayer}
		\draw [bend left=90, looseness=1.75] (21.center) to (22.center);
		\draw [bend right=90, looseness=1.75] (21.center) to (22.center);
		\draw [bend left=90, looseness=1.75] (29.center) to (30.center);
		\draw [bend right=90, looseness=1.75] (29.center) to (30.center);
		\draw [style=open, bend left=15, looseness=0.75] (32.center) to (33.center);
	\end{pgfonlayer}
\end{tikzpicture}    \right) =\xi_{P,Q} 
	\end{equation}
This proves~\ref{tracefieldi}.

	For the proof of \ref{tracefieldii}, recall that in 
 the setting of symmetric Frobenius algebras, it is shown in \cite[page~128]{kock} that the handle element is central. 
	A straightforward computation
	by means of the trace field theory $\Phi_\cat{C}$ shows that this holds still true in our more general situation: In fact, one can directly see that both the map $\xi_{P,Q}\circ - : \cat{C}(P,P)\to\cat{C}(P,P)$
	and the map $-\circ \xi_{P,Q}: \cat{C}(P,P)\to\cat{C}(P,P)$ are  given by
	\begin{align}
\Phi_\cat{C}\left(\begin{tikzpicture}
	\begin{pgfonlayer}{nodelayer}
		\node [style=none] (165) at (34.25, 29) {};
		\node [style=none] (166) at (34.25, 26) {};
		\node [style=none] (167) at (32.25, 27.75) {{\footnotesize $P$}};
		\node [style=none] (168) at (34.25, 28) {};
		\node [style=none] (169) at (34.25, 27) {};
		\node [style=none] (170) at (35.25, 27.5) {{\footnotesize $Q$}};
		\node [style=none] (173) at (34.25, 25.5) {\IN};
		\node [style=none] (174) at (35.5, 28.5) {};
		\node [style=none] (175) at (33, 28.5) {};
		\node [style=none] (176) at (34.25, 29.5) {\OUT};
		\node [style=none] (177) at (34.25, 27) {};
		\node [style=none] (178) at (33, 26.5) {};
		\node [style=none] (179) at (35.5, 26.5) {};
	\end{pgfonlayer}
	\begin{pgfonlayer}{edgelayer}
		\draw [bend left=90, looseness=1.75] (165.center) to (166.center);
		\draw [bend right=90, looseness=1.75] (165.center) to (166.center);
		\draw [bend left=90, looseness=1.75] (168.center) to (169.center);
		\draw [bend right=90, looseness=1.75] (168.center) to (169.center);
		\draw [style=open, bend right=45] (174.center) to (175.center);
		\draw [style=open, bend right=45] (178.center) to (179.center);
	\end{pgfonlayer}
\end{tikzpicture}\right) 
		\end{align}
	in terms of the trace field theory.

	For the proof of~\ref{tracefieldiii}, first observe
	\begin{equation}
	\Phi_\cat{C} \left( \begin{tikzpicture}[scale=0.7]
	\begin{pgfonlayer}{nodelayer}
		\node [style=none] (21) at (0, 2) {};
		\node [style=none] (22) at (0, -1) {};
		\node [style=none] (25) at (-1.5, 1) {};
		\node [style=none] (26) at (-1.5, 0) {};
		\node [style=none] (27) at (2, 0.5) {{\small $P$}};
		\node [style=none] (28) at (1, 2) {};
		\node [style=none] (29) at (-2.25, 0.5) {\IN};
	\end{pgfonlayer}
	\begin{pgfonlayer}{edgelayer}
		\draw [bend left=90, looseness=1.75] (21.center) to (22.center);
		\draw [bend right=90, looseness=1.75] (21.center) to (22.center);
		\draw [style=open, bend right=15, looseness=0.75] (25.center) to (26.center);
	\end{pgfonlayer}
\end{tikzpicture}  \right) \circ  \Phi_\cat{C}\left(    \begin{tikzpicture}
	\begin{pgfonlayer}{nodelayer}
		\node [style=none] (21) at (0, 2) {};
		\node [style=none] (22) at (0, -1) {};
		\node [style=none] (28) at (-2, 0.75) {{\small $P$}};
		\node [style=none] (29) at (0, 1) {};
		\node [style=none] (30) at (0, 0) {};
		\node [style=none] (31) at (0, -0.5) {{\small $Q$}};
		\node [style=none] (32) at (1.5, 1) {};
		\node [style=none] (33) at (1.5, 0) {};
		\node [style=none] (34) at (2.25, 0.5) {\OUT};
	\end{pgfonlayer}
	\begin{pgfonlayer}{edgelayer}
		\draw [bend left=90, looseness=1.75] (21.center) to (22.center);
		\draw [bend right=90, looseness=1.75] (21.center) to (22.center);
		\draw [bend left=90, looseness=1.75] (29.center) to (30.center);
		\draw [bend right=90, looseness=1.75] (29.center) to (30.center);
		\draw [style=open, bend left=15, looseness=0.75] (32.center) to (33.center);
	\end{pgfonlayer}
\end{tikzpicture}    \right) =      \Phi_\cat{C}\left(    \begin{tikzpicture}
	\begin{pgfonlayer}{nodelayer}
		\node [style=none] (21) at (0, 2) {};
		\node [style=none] (22) at (0, -1) {};
		\node [style=none] (28) at (0, 2.5) {{\small $P$}};
		\node [style=none] (29) at (0, 1) {};
		\node [style=none] (30) at (0, 0) {};
		\node [style=none] (31) at (0, -0.5) {{\small $Q$}};
	\end{pgfonlayer}
	\begin{pgfonlayer}{edgelayer}
		\draw [bend left=90, looseness=1.75] (21.center) to (22.center);
		\draw [bend right=90, looseness=1.75] (21.center) to (22.center);
		\draw [bend left=90, looseness=1.75] (29.center) to (30.center);
		\draw [bend right=90, looseness=1.75] (29.center) to (30.center);
	\end{pgfonlayer}
\end{tikzpicture}    \right)=\dim\,\cat{C}(P,Q) \ . 
	\end{equation}
	(This is the generalization of the fact that the counit evaluated on the handle element on a symmetric Frobenius algebra
	is the linear dimension \cite[page~129]{kock}  to Calabi-Yau categories.) 
	Now we use Theorem~\ref{thmtracefieldtheory} which 
	asserts that the evaluation of $\Phi_\cat{C}$ on the disk with one incoming open boundary interval is actually the modified trace. This proves \eqref{eqntraceformula}. 
\end{proof}

Let us also formulate this on the level of homology and denote for an endomorphism $f: P\to P$ of $P\in\Proj\cat{C}$
  by $\HS(f)\in HH_0(\cat{C})$ the Hattori-Stallings trace. Then Theorem~\ref{thmtracefieldtheory} and Theorem~\ref{thmhandleelement} imply:

	\begin{corollary}\label{corhs}
		For any finite tensor category $\cat{C}$ with symmetric Frobenius structure,
		\begin{align}
			\trace (\HS (\id_P) \star \HS (\id_Q) ) = \dim\, \cat{C}(P,Q) \quad \text{for}\quad P,Q \in \Proj \cat{C} \ . 
			\end{align}
		\end{corollary}
	Here, by slight abuse of notation,
	 we denote the map on $HH_0(\cat{C})$
	induced by the modified trace 
	again by $\trace$.

\spaceplease

\end{document}